\documentclass[11pt]{amsart}
\usepackage{amsmath,amsthm,amssymb,latexsym,color}
\usepackage{hyperref}
\usepackage{lipsum}
\usepackage{todonotes}
\usepackage{multirow}
\usepackage[mathscr]{eucal}

\usepackage{xcolor}
\definecolor{b}{HTML}{4472c4}
\definecolor{o}{HTML}{ED7D31}
\definecolor{g}{HTML}{70ad47}
\definecolor{t}{RGB}{0,0,255}
\newcommand\co[1]{{\color{o}#1}}

\usepackage{xcolor}
\date{}
\usepackage{pgfplots}
\usepackage{tikz}
\usetikzlibrary{arrows,matrix,positioning}
\usetikzlibrary{decorations.pathreplacing,decorations.pathmorphing,calc}
\usepackage{graphicx}
\newtheorem{theor}{Theorem}[section]
\newtheorem*{theorem*}{Theorem}

\newtheorem{lemma}[theor]{Lemma}
\newtheorem{cor}[theor]{Corollary}
\newtheorem{defi}[theor]{Definition}
\newtheorem{prop}[theor]{Proposition}

\theoremstyle{definition}

\theoremstyle{plain}

\newcommand{\N}{\mathbb{N}}

\newcommand{\R}{\mathbb{R}}

\def\Prob{{\mathbb P}}
\def\Reg{{\mathcal R}}
\def\Rel{{\mathscr{R}}}

\newcommand{\imperfect}{\mathcal{I}_{n,d}}
\newcommand{\tupleset}{\mathcal{H}}

\newcommand{\weight}{\mathcal{W}}

\newcommand{\length}[1]{{\rm len}(#1)}

\def\dist{{\rm dist}}

\def\path{{\mathcal P}}
\def\tuple{{\mathcal T}}
\def\mult{{\rm mult\,}}

\def\cconst{{\mathfrak m}}

\title{Regularized modified log-Sobolev inequalities, and comparison of Markov chains}
\author{
Konstantin Tikhomirov
\and
Pierre Youssef
}
\thanks{K.T. is partially supported by the Sloan Research Fellowship
and by the NSF grant DMS 2054666}
\address{
\medskip
\noindent
Konstantin Tikhomirov,
School~of Math.,
GeorgiaTech,
686 Cherry street,
Atlanta, GA 30332.\\
\texttt{\small
e-mail:   ktikhomirov6@gatech.edu}
}
\address{
\medskip
\noindent
Pierre Youssef, 
Division of Science, NYU Abu Dhabi, Saadiyat Island, Abu Dhabi, UAE \& Courant
Institute of Mathematical Sciences, New York University, 251 Mercer st, New York,
NY 10012, USA.\\
\texttt{\small
e-mail:  yp27@nyu.edu}}

\def\R{{\mathbb R}}
\def\N{{\mathbb N}}

\newcommand{\Perfmatch}{\mathcal{C}_{n,d}}

\newcommand{\Ugly}{\mathcal{U}_{n,d}}

\def\BipGSet{\Omega^B}

\def\ConfBipGSet{\Omega^{B\,C}}
\def\SNeigh{{\mathcal S\mathcal N}}
\def\categ{{\rm Cat}_{n,d}}

\def\tmix{{t_{mix}}}

\def\Prob{{\mathbb P}}
\def\Exp{{\mathbb E}}

\def\Dir{{\mathcal E}}

\newcommand{\Ent}{{\rm Ent}}
\newcommand{\Var}{{\rm Var}}
\newcommand{\Psimple}{\pi_{u}}
\newcommand{\Pconfig}{\pi_{BC}}

\def\neigh{{\mathcal N}}
\def\dist{{\rm dist}}
\def\cov{{\rm Cov}}

\begin{document}

\maketitle

\begin{abstract}
In this work, we develop a comparison procedure for the Modified log-Sobolev Inequality
(MLSI) constants 
of two reversible Markov chains on a finite state space. 
Efficient comparison of the MLSI Dirichlet forms
is a well known obstacle in the theory of Markov chains.
We approach this problem by introducing a {\it regularized} MLSI
constant
which, under some assumptions, has the same order of magnitude
as the usual MLSI constant yet
is amenable for comparison and thus considerably simpler to estimate in certain cases.
As an application of this general comparison procedure, we
provide a sharp estimate of the MLSI constant 
of the switch chain on the the set of simple bipartite regular graphs of size $n$ with a fixed degree $d$. Our estimate 
implies that the total variation mixing time of the switch chain is
of order $O_d(n\log n)$. The result is optimal up to a multiple
depending on $d$ and resolves a long-standing open problem. 
We expect that the MLSI comparison technique
implemented in this paper will find further applications.
\end{abstract}

\section{Introduction}

Let $\Omega$ be a finite state space, and let $Q$ be a Markov generator of a reversible chain on $\Omega$
with a stationary distribution $\pi$. 
We say that 
$(\Omega, \pi, Q)$ satisfies a {\it Modified Logarithmic Sobolev Inequality (MLSI)} with a constant $\alpha$
if for any function $f:\, \Omega\to (0,\infty)$ we have
$$
{\rm Ent}_\pi(f):= \Exp_\pi\big[f(\log f- \log \Exp_\pi f)\big] \leq \alpha\, \Dir_\pi(\log f,f),
$$
where $\Dir_\pi(\log f,f)= \frac12\sum_{\omega,\omega'\in \Omega} \pi(\omega) Q(\omega, \omega') \big(f(\omega)-f(\omega')\big)\log\frac{f(\omega)}{f(\omega')}$ 
is the corresponding {\it Dirichlet form}\footnote{Often, the MLSI constant is defined as inverse of the above; however, we prefer to use the given definition.}. We refer to the smallest $\alpha$ in the above inequality as the MLSI constant and denote it by $\alpha_{MLSI}(Q)$.

Similarly to the {\it log-Sobolev inequality} ${\rm Ent}_\pi(f^2)
\leq \alpha_{LSI}\, \Dir_\pi(f,f)$, the MLSI
is known to imply sub-Gaussian 
concentration via the Herbst argument (see, for example, \cite[Chapter~5]{ledoux}). Moreover, it constitutes a powerful 
tool allowing to capture the mixing time of the underlying Markov chain (see \cite{bobkov-tetali}). More precisely, for every $\varepsilon \in (0,1)$,
\begin{equation}\label{eq: mix-mlsi}
\tmix(Q,\varepsilon)\leq  \alpha_{MLSI}(Q)\big( \log\log \frac{1}{\pi_{\min}}+\log\frac{1}{2\varepsilon^2}\big),
\end{equation}
where $\tmix(Q,\varepsilon)$ denotes the total variation $\varepsilon$-mixing time of $Q$ and $\pi_{\min}=\min_{x\in \Omega} \pi(x)$.% and $\alpha_{MLSI}(Q)$ is the MLSI constant.  
While sharing similar properties with the log-Sobolev inequality, the MLSI
often holds with a much smaller
constant than the log-Sobolev inequality,
and thus allows to get stronger
concentration and mixing estimates. 

Estimating the 
relaxation time or the log-Sobolev constant of a Markov chain by comparing
it with another random process is a well developed technique 
which has been
successfully used in a variety of situations 
(see for instance \cite{DS-Comparison, DS-Comparison2,saloff,Dyer}
as well as a recent paper \cite{TY2020+} by the authors for details).  
The main idea is that 
when the stationary measures of two Markov chains
are %''comparable''
``close'' to each other, 
the Poincar\'e and log-Sobolev constants
of the chains can be related by 
comparing the corresponding 
Dirichlet forms of the two chains. The {\it canonical path} (or the {\it flow})
method \cite{Sinclair} aims at providing 
an efficient relation between the Dirichlet forms. 
While this comparison procedure has been widely used
to obtain bounds on the Poincar\'e and the log-Sobolev constants, the case 
of the MLSI constant turns out to be fundamentally different. 
Indeed, let $(\pi, Q)$ and $(\tilde \pi, \tilde Q)$ be
two reversible irreducible Markov generators on a finite state space $\Omega$.
It is known (see \cite[Chapter~4]{saloff}) that
there exists a constant $C$ (depending on $Q$, $\tilde Q$)
such that for any $f:\, \Omega\to \R_+$ one has 
$\tilde \Dir(f,f)\leq C\, \Dir(f,f)$ where $\tilde \Dir$ (resp. $\Dir$) denotes the Dirichlet form associated with $(\tilde \pi, \tilde Q)$
(resp. $(\pi, Q)$). 
%Such a comparison for the Dirichlet forms is the key to obtain relations between  the Poincar\'e and log-Sobolev of the 
% corresponding chains.
On the other hand, under the same assumptions, there {\it does not}
in general exist a constant $C$ such that for all $f:\, \Omega\to \R_+$ one has
$\tilde \Dir(f,\log f)\leq C\, \Dir(f,\log f)$ (see \cite[Page~74]{G} for a counter-example).
%This prevented the elaboration of a general comparison mechanism for the modified log-Sobolev constant. 

%is shortcoming and elaborate a comparison procedure for MLSI constants. 
In this paper, we develop a comparison procedure for the MLSI constants
based on a notion of a {\it regularized Modified log-Sobolev Inequality},
which is the MLSI restricted to a special class of functions.
Given a Markov chain on $\Omega$, we show that
the MLSI and its regularized version hold with constants having the same order
of magnitude. We then show that under certain assumptions the
Dirichlet forms of two Markov chains evaluated on the regular functions
can be efficiently compared.

%The main observation of this paper is that the comparison of the Dirichlet forms mentioned above is not needed for \textbf{all} functions but rather for 
%a class of ``regular'' functions in some sense. Moreover, it turns out that there exists a constant $C$ such that for all such regular functions, one has 
%  $\tilde \Dir(f,\log f)\leq C\, \Dir(f,\log f)$. This motivates us to introduce the notion of regularized MLSI which we will discuss below. 
  
Below, we provide a rigorous description of our method.   
Given a reversible Markov generator $Q$ on a 
state space $\Omega$ with a stationary distribution $\pi$, 
we equip $\Omega$ 
with the graph structure induced by $Q$, namely,
%the vertex set is $\Omega$ and
two distinct vertices $\omega, \omega'\in \Omega$ 
 are connected by an edge if and only if $Q(\omega,\omega')\neq 0$. 
 Given $r\geq 1$, define $\Reg(Q,r)$
as the collection of all functions $f:\Omega\to(0,\infty)$ such that
$$
f(\omega)/f(\omega')\leq r^{\dist(\omega,\omega')}\quad\mbox{for all vertices $\omega,\omega'$ of $\Omega$},
$$
where $\dist(\omega,\omega')$ is the usual graph distance between $\omega$ and $\omega'$. 
We call these functions {\it $r$--regular}. Note that if a function is $r$-regular then it is also $r'$-regular for any $r'\geq r$. 
Moreover, any positive function is $\infty$-regular while constant functions are $1$-regular. 

We say that $(\Omega, \pi, Q)$ satisfies the {\it $r$-regularized MLSI} with a constant $\alpha_{r}$ if for  
any function $f\in \Reg(Q,r)$ we have 
$$
{\rm Ent}_\pi(f)\leq \alpha_{r}\, \Dir_\pi(\log f,f).
$$
As before, we refer to the best constant in the above inequality as the $r$-regularized MLSI constant of $Q$.
Note that with our notations, the $\infty$-regularized MLSI is the ``usual''
Modified log-Sobolev inequality, and in view of the above, whenever
$(\Omega, \pi, Q)$ satisfies the ``usual'' MLSI,
it also satisfies the $r$-regularized MLSI with the same constant for any $r\geq 1$. 
Our first main result shows that there exists $1<r<\infty$ for which the reverse is true. 

\begin{theor}\label{th: MLSI=reg-MLSI}
Let $Q$ be a reversible Markov generator with a stationary measure $\pi$ on a finite state space $\Omega$. 
Define
\begin{equation}\label{eq: m 05623058}
\gamma:=\frac{\max_{\omega\in \Omega} \pi(\omega)}{\min_{\omega\in \Omega} \pi(\omega)}\quad \text{and}\quad 
\Upsilon:=\frac{16\gamma^2\max\limits_{\omega}|\{\omega'\neq \omega:\;Q(\omega,\omega')\neq 0\}|}
{\min\limits_{\omega\neq\omega': Q(\omega,\omega')\neq 0}Q(\omega,\omega')}.
\end{equation}
If $(\Omega, \pi, Q)$ satisfies the $\Upsilon$-regularized MLSI with a constant $\alpha_{\Upsilon}$, then 
$(\Omega, \mu, Q)$ satisfies MLSI with constant $C\alpha_{\Upsilon}$, where $C>0$ is a universal constant.
\end{theor}

While the above is satisfactory for the application we have in mind, it would be interesting to 
find the ``best'' value of the parameter $r$ for which the $r$-regularized MLSI implies MLSI. 
We did not pursue this line of research in the current work. %as our motivation lies elsewhere 
%and focuses on illustrating the utility of this notion. 
%Indeed, as we will see next, the regularized MLSI allows to implement a comparison procedure between 
%two Markov chains. 

Let $(\pi, Q)$ and $(\tilde \pi, \tilde Q)$ be two reversible Markov generators on a finite set $\Omega$. 
For each $x,y \in \Omega$ with $\tilde Q(x,y)>0$, we let $\path_{x,y}$ be the set of all paths $x_0=x, x_1,\ldots, x_k=y$ (of arbitrary lengths $k\geq 1$)
such that $Q(x_i,x_{i+1})>0$ for all $i=0,\ldots ,k-1$. 
We define $\Gamma(Q,\tilde Q):=\bigcup_{x,y\in \Omega,\, \tilde Q(x,y)>0} \path_{x,y}$.
Recall that a weight function $\weight:\, \Gamma(Q,\tilde Q)\to [0,1]$ is called a {\it $(Q,\tilde Q)$--flow} if for every $x,y$ with $\tilde Q(x,y)>0$ we have
$$
\sum_{P\in \path_{x,y}} \weight(P)= \tilde \pi(x) \tilde Q(x,y)
$$
(see \cite[Section~2C]{DS-Comparison2}).
The second main result in the paper is the following theorem. 

\begin{theor}\label{th: comparison-mlsi}
Let $(\pi, Q)$ and $(\tilde \pi, \tilde Q)$ be two reversible Markov generators on a finite set $\Omega$, let $\weight$ be
a $(Q,\tilde Q)$--flow, and suppose that $\pi(\omega)\leq a \tilde \pi(\omega)$ for every $\omega\in \Omega$ for some parameter $a>0$. 
If $(\Omega, \tilde \pi, \tilde Q)$ satisfies MLSI with constant $\tilde \alpha(\tilde Q)$, then for any $r\geq 2$ 
the $r$-regularized modified log-Sobolev constant $\alpha_r(Q)$ of $(\Omega, \pi, Q)$ satisfies 
$$
\alpha_r(Q)\leq C\,a\,A(\weight,r)\,\tilde \alpha_r(\tilde Q)\, , 
$$
where
\begin{equation}\label{nfaofjnpifqwnpijnfijnw}
A(\weight,r)= \max_{\underset{Q(\omega,\omega')>0}{(\omega, \omega')}} \frac{1}{\pi(\omega)Q(\omega, \omega')} 
\sum_{\underset{(\omega,\omega')\in \path}{\path\in \Gamma(Q, \tilde Q)}} \weight(\path)  \big(1+ (\length{\path}-1)^2\log r\big),
\end{equation}
and $C$ is a universal constant. 
\end{theor}

Note that without imposing the $r$--regularization on functions on $\Omega$
(i.e when considering the setting $r=\infty$), the above comparison
result in itself does not produce a useful estimate.
However, when combined with Theorem~\ref{th: MLSI=reg-MLSI},
a comparison statement for the ``usual'' MLSI constants readily follows. 
This fills a gap in 
the literature by providing  a result for MLSI similar  to classical comparison statements for Poincar\'e and log-Sobolev inequalities (see for example \cite[Chapter~4]{saloff}).
%In view of the dependence on $r$, the above theorem already justifies the absence 
%in the literature of any comparison procedure for the modified log-Sobolev constant (which corresponds to $r=\infty$). 
%However, when combined with Theorem~\ref{th: MLSI=reg-MLSI}, such comparison statement readily follows. 

The above result becomes interesting when the MLSI and the log-Sobolev constants of $(\tilde \pi, \tilde Q)$ have different orders of magnitude.
Indeed,
it is always 
possible to bound $\alpha_r(Q)$ by the log-Sobolev constant $\alpha_{LSI}(Q)$
and then use standard comparison procedures (see, in particular, \cite[Theorem~4.2.5]{saloff}) to bound the latter by the 
log-Sobolev constant of $(\tilde\pi,\tilde Q)$
multiplied by a function of the flow similar to the one in Theorem~\ref{th: comparison-mlsi}.  
%Such procedure may not always be possible since the chain $\tilde Q$ may not satisfy a log-Sobolev inequality or its constant may 
%be too large.
One particular example is when $\tilde Q$ is the trivial Markov generator on $(\Omega, \pi)$, with $\tilde Q(\omega,\omega')=\pi(\omega')=\tilde\pi(\omega')$
for any $\omega\neq \omega'$.
It is known that this chain satisfies a log-Sobolev inequality
with $\alpha_{LSI}(\tilde Q)=O\big(\log\frac{1}{\min_{\omega\in \Omega}\pi(\omega)}\big)$
and the Modified log-Sobolev Inequality
with constant $1$. 
%Using this trivial log-Sobolev inequality,
It follows (see \cite[Theorm~2.3]{DS-Comparison2}
and \cite[Section~4.2]{saloff}) that 
any triple $(\Omega, \pi, Q)$ satisfies a log-Sobolev inequality (and thus MLSI) with a constant 
$$
C\log \frac{1}{\min_{\omega\in \Omega}\pi(\omega)} \max_{\underset{Q(\omega,\omega')>0}{(\omega, \omega')}} \frac{1}{\pi(\omega)Q(\omega, \omega')} 
\sum_{\underset{(\omega,\omega')\in \path}{\path\in \path_{x,y},\, x\neq y}} \weight(\path) \length{\path}.
$$
The bound provided by combining Theorems~\ref{th: MLSI=reg-MLSI} and~\ref{th: comparison-mlsi} improves the last estimate in many situations of interest, 
as it replaces the ``global''
factor $\log \frac{1}{\min_{\omega\in \Omega}\pi(\omega)}$ by a ``local'' parameter $\log \Upsilon$, at
the price of squaring the lengths of the paths in the flow. 

To illustrate the power of the comparison procedure introduced in this paper, we will apply this concept to derive 
a sharp Modified log-Sobolev Inequality for the {\it switch chain} on the set of regular bipartite graphs. 
This chain uses a standard local operation called the
{\it simple switching} which takes two non-incident edges $(i_1,j_1)$
and $(i_2,j_2)$ of the graph uniformly at random, destroys them, 
and replaces them by their ``crossed'' counterparts $(i_1,j_2)$
and $(i_2,j_1)$ whenever possible. Formally, given $n\in \N$ and $2\leq d\leq n/2$, we denote by $\BipGSet_n(d)$
the set of all simple bipartite $d$--regular graphs on
the vertex set
$[n^{(\ell)}]\sqcup[n^{(r)}]$ (where we use the superscripts
``$(\ell)$'' and ``$(r)$''
for sets of left and right vertices),
and we equip it with the uniform probability measure $\Psimple$. 
The switch chain is defined through its Markov generator $Q_u$ as follows:
$$
Q_u(G_1,G_2):=\begin{cases}
-\frac{|\neigh(G_1)|}{nd(nd-1)/2},&\mbox{if $G_1=G_2$};\\
\big(nd(nd-1)/2\big)^{-1},&\mbox{if $G_2\in \neigh(G_1)$};\\
0,&\mbox{otherwise.}
\end{cases}
$$
Here, $\neigh(G_1)$ denotes the set of all graphs in $\BipGSet_n(d)$ which can be obtained
from $G_1$ by the simple switching operation. 
%While originally motivated by the problem of generating regular graphs, the study of this chain became important on its own. 
The mixing time of this chain was first investigated in \cite{KTV},
followed by papers \cite{CDG,Greenhill,Greenhill2,MES,EKMS,EMMS,BHS,AK,GS,unified}
which studied the switch chain 
for several graph models. We refer to \cite{unified} for a recent account of this line of research and a comprehensive reference list. 

Recently, the authors \cite{TY2020+} established a sharp Poincar\'e inequality for the chain $(\BipGSet_n(d), \Psimple, Q_u)$ for any degree $d\geq 3$ satisfying
$d\leq n^{c}$, for some small universal constant $c$. 
When $d$ is fixed, they also established a log-Sobolev inequality with a constant $C_dn\log n$ and showed that the dependence of the LSI constant on $n$ is sharp. 
The strategy employed in \cite{TY2020+} is a double comparison procedure with the standard random transposition model and the switch chain 
on the configuration model. The main challenge in \cite{TY2020+} is that in the regime
$d\to \infty$, the configuration model and the space
$(\BipGSet_n(d), \Psimple, Q_u)$ do not
admit
%differ significantly 
%and thus the comparison strategy does not fall into the category of
standard comparison techniques for Markov chains without incurring a loss of precision.
To overcome this issue, 
a delicate construction of function extensions with induced
``controlled'' fluctuations was introduced \cite{TY2020+}.
%and an explicit notion of harmonic functions and Gaussian Free Field was used. 
When $d$ is fixed, the standard comparison techniques can be employed,
and the main technical task is to construct {\it canonical path} and 
a flow with a small congestion. This was carried out
in \cite{TY2020+} and allowed the authors to obtain the sharp log-Sobolev inequality 
which implied in particular that the total variation mixing time of the switch chain is bounded above by $C_d n\log^2n$ for some constant depending only on $d$. 
Previously, the best known bound in this regime was $C_dn^7\log n$ \cite{Dyer-arxiv}. 
The mixing time bound obtained in \cite{TY2020+} is off by a factor $\log n$ from the conjectured optimal estimate.
That in itself is not surprising since the approach relied on 
a comparison with the random transposition model, and it is known that the log-Sobolev constant for that model fails to capture the correct total variation mixing time \cite{G}. 
On the other hand, the sharp MLSI constant for the random transposition model was calculated in \cite{G} and it was shown that it does yield
the optimal TV mixing time bound. 
This is one particular instance where the modified log-Sobolev inequality offers an advantage over the classical log-Sobolev inequality. To summarize, the comparison 
techniques developed in this paper %become of great use and
allowed us to prove the following.
\begin{theor}\label{th: mlsi-switch}
For every fixed $2\leq d\leq n/2$, the triple
$(\BipGSet_n(d), \Psimple, Q_u)$ satisfies the Modified log-Sobolev Inequality with a
constant $C_dn$, where $C_d>0$ depends only on $d$. 
\end{theor}

%As we will show in Section~\ref{sec: mlsi-switch}, the above estimate is sharp. 
%and improves on the bound $C_dn\log n$ which one would obtain as 
%a consequence of the log-Sobolev inequality established in \cite{TY2020+}. 
%Moreover, in view of \eqref{eq: mix-mlsi}, this result allows to capture the correct bound on the mixing time. 

\begin{cor}
For every fixed $2\leq d\leq n/2$, the total variation
mixing time $\tmix:=\tmix(Q_u, \frac14)$ of the switch chain
$(\BipGSet_n(d), \Psimple, Q_u)$ is bounded above by $C_dn\log n$,
for some constant $C_d$ depending only on $d$. 
\end{cor}

The above bound is sharp (see Proposition~\ref{prop:lower bound}) and was previously conjectured in \cite{CDG}. With the techniques developed in \cite{TY2020+} and the present paper,
we believe it is possible to derive sharp bounds on the mixing time of the switch chain for other graph models of interest, including simple undirected $d$--regular graphs. 

The paper is organized as follows: Section~\ref{sec: regularization} is devoted for the proof of Theorem~\ref{th: MLSI=reg-MLSI}, while Theorem~\ref{th: comparison-mlsi} is proved in Section~\ref{sec: mlsi-comparison}. Finally, the proof of Theorem~\ref{th: mlsi-switch} is carried in Section~\ref{sec: mlsi-switch}. 

\section{MLSI and function regularization}\label{sec: regularization}

Before proceeding with the proof of Theorem~\ref{th: MLSI=reg-MLSI}, we consider %would like to put forward
another statement in the same spirit aiming at restricting the class of functions on which the MLSI needs be verified. The next lemma may be of independent interest, and will be used in Section~\ref{sec: mlsi-switch} when proving the MLSI for the switch chain. 

\begin{lemma}\label{akjfnqoifnfoqifalkdjfn}
There are universal constants $c,C>0$ with the following property.
Assume that a reversible Markov chain $(\Omega,Q,\pi)$ satisfies
$$
\Ent_\pi\, f\leq K\,\Dir_\pi(f,\log f)
$$
for every positive function $f$ on $\Omega$ with $f(\omega)\geq c$ for all
$\omega\in\Omega$ and $\Exp_\pi f=1$.
Then $(\Omega,Q,\pi)$ satisfies the MLSI with a constant $CK$.
\end{lemma}
\begin{proof}
We will assume that the constant $c$ is sufficiently small so that in particular $1-c+c\log c\geq \frac{1}{2}$
and $c/(1-c)\leq 1/2$.

Fix any non-constant positive function $f:\Omega\to(0,\infty)$ with $\Exp_\pi f=1$, and define
an auxiliary function $f'$ as follows:
$$
f'(\omega):=\begin{cases}
\max(f(\omega),c),&\mbox{ if $f(\omega)\leq 1$};\\
\alpha +(1-\alpha)f(\omega)&\mbox{ if $f(\omega)> 1$},
\end{cases}
$$
where the parameter $\alpha\in[0,1)$ is chosen so that
$\Exp_\pi f'=1$. 
Note that in view of our assumptions,
$$
\Ent_\pi\, f'\leq K\,\Dir_\pi(f',\log f').
$$

\medskip

We first estimate the value of the parameter $\alpha$.
Let $T_0,T_{[c,1]},T_{>1}$ be the partition of the space $\Omega$
into subsets of points $\omega$ where
$f(\omega)<c$, $f(\omega)\in[c,1]$ and $f(\omega)>1$, respectively.
Thus,
$$
1=\Exp_\pi f'=c\,\pi(T_0)+\sum_{\omega\in T_{[c,1]}}f(\omega)\pi(\omega)
+\alpha\,\pi(T_{>1})+(1-\alpha)\sum_{\omega\in T_{>1}}f(\omega)\pi(\omega),
$$
implying that
$$
\alpha\bigg(\sum_{\omega\in T_{>1}}f(\omega)\pi(\omega)-\pi(T_{>1})\bigg)
=c\,\pi(T_0)-\sum_{\omega\in T_0}f(\omega)\pi(\omega).
$$
It remains to observe that
$$
\sum_{\omega\in T_{>1}}f(\omega)\pi(\omega)-\pi(T_{>1})
=\pi(T_0)-\sum_{\omega\in T_0}f(\omega)\pi(\omega)+
\pi(T_{[c,1]})-\sum_{\omega\in T_{[c,1]}}f(\omega)\pi(\omega)
\geq (1-c)\pi(T_0)
$$
to conclude that $\alpha\leq \frac{c}{1-c}\leq 1/2$.

\medskip

The next step of the argument is to compare the entropies of the functions $f$ and $f'$.
We will use the representations of the entropies
\begin{align*}
\Ent_\pi \,f&=\sum_{\omega\in\Omega}\big(1-f(\omega)+f(\omega)\log f(\omega)\big)\pi(\omega);\\
\Ent_\pi \,f'&=\sum_{\omega\in\Omega}\big(1-f'(\omega)+f'(\omega)\log f'(\omega)\big)\pi(\omega),
\end{align*}
which have the advantage that the convex function $x\to 1-x+x\log x$, $x\in(0,\infty)$,
is non-negative, allowing term-by-term comparison of the expressions on the right side.
Clearly, for every $\omega\in T_{[c,1]}$, the respective terms agree.
Further, for any $\omega\in T_0$, in view of the conditions on $c$,
$$
1-f'(\omega)+f'(\omega)\log f'(\omega)
=1-c+c\log c\geq \frac{1}{2}\geq\frac{1}{2}\big((1-f(\omega)+f(\omega)\log f(\omega)\big),
$$
where we also used that $x\to 1-x+x\log x$ is bounded above by $1$ on $(0,1)$. 
Finally, for $\omega\in T_{>1}$ we consider two cases. If $f(\omega)\geq 10$ then $f'(\omega)\geq f(\omega)/2\geq 5$,
%and both expressions $1-f(\omega)+f(\omega)\log f(\omega)$ and $1-f'(\omega)+f'(\omega)\log f'(\omega)$ are ``dominated'' by the terms with the logarithm, that is,
and we have
$$
1-f(\omega)+f(\omega)\log f(\omega)\leq f(\omega)\log f(\omega),\quad
1-f'(\omega)+f'(\omega)\log f'(\omega)\geq \frac12 f'(\omega)\log f'(\omega).
$$
At the same time, since $f'(\omega)\geq f(\omega)/2$ and $f(\omega)\geq 10$, we have 
$f'(\omega)\log f'(\omega)\geq \frac{1}{4}f(\omega)\log f(\omega)$. Thus, whenever
$f(\omega)\geq 10$, we have $1-f'(\omega)+f'(\omega)\log f'(\omega)\geq \frac18\big(
1-f(\omega)+f(\omega)\log f(\omega)\big)$. 
%for some universal constant $c_1>0$.
In the remaining case $f(\omega)\in(1,10)$, we observe that 
$$
1-f(\omega)+f(\omega)\log f(\omega)\leq (f(\omega)-1)^2,\quad
1-f'(\omega)+f'(\omega)\log f'(\omega)\geq \frac16 (f'(\omega)-1)^2,
$$
and thus we conclude that
$1-f'(\omega)+f'(\omega)\log f'(\omega)\geq \frac{(1-\alpha)^2}{6}\big(
1-f(\omega)+f(\omega)\log f(\omega)\big)$.

To summarize, we have shown that
$$
C\,\Ent_\pi \,f'\geq \Ent_\pi \,f
$$
for some constant $C>0$.

\medskip

Now, we compare the Dirichlet forms with the functions $f'$ and $f$.
This step is elementary
since it is sufficient for us to confirm that $\Dir_\pi (f',\log f')\leq \Dir_\pi (f,\log f)$, while
the construction of $f'$ guarantees that for every $\omega,\omega'\in\Omega$ with
$f(\omega)\leq f(\omega')$, we have $f'(\omega')-f'(\omega)\leq f(\omega')-f(\omega)$
and $\frac{f'(\omega')}{f'(\omega)}\leq \frac{f(\omega')}{f(\omega)}$.
The result follows.
\end{proof}

The main goal of this section is to prove Theorem~\ref{th: MLSI=reg-MLSI}.
We first define the notion of $r$--regularization. 
Given a reversible Markov generator $Q$ on $(\Omega, \pi)$, and a positive function $f:\Omega\to \R_+$, 
the {\it $r$--regularization} of $f$ is the function $f_{r}$ given by
$$
f_{r}(\omega):=\max\limits_{\omega'\in \Omega}\frac{f(\omega')}{r^{\dist(\omega,\omega')}}\quad\mbox{for every $\omega\in \Omega$},
$$
where $\dist(\cdot,\cdot)$ is the usual distance in the
graph $(\Omega,\{(w,w'):\,w\neq w',\,Q(w,w')\neq 0\})$. 

In what follows, given $r>1$, it will be convenient to associate with every positive function $f$ on $\Omega$
a mapping $\Rel_{f}=\Rel_{f,\Omega,r}$ as follows.
Let $f_{r}$ be the $r$--regularization of $f$. For every $\omega\in \Omega$ with $f_{r}(\omega)>f(\omega)$
there is at least one vertex $\tilde \omega\in\Omega$ such that $f_{r}(\omega)=\frac{f(\tilde \omega)}{r^{\dist(\omega,\tilde\omega)}}
=\frac{f_{r}(\tilde \omega)}{r^{\dist(\omega,\tilde\omega)}}$. Then we set $\Rel_{f}(\omega):=\tilde\omega$. Thus, $\Rel_{f}$ is a mapping
on $\{\omega:\;f_{r}(\omega)>f(\omega)\}$.
Note that in general $\Rel_{f}$ does not have to be uniquely defined.
For convenience, we will fix a single realization of $\Rel_{f,\Omega,r}$ for every triple $(f,\Omega,r)$.

Simple properties of $r$--regularizations are collected in the following lemma.
\begin{lemma}\label{l: 317610498}
Let $Q$ be a reversible Markov generator on a finite probability space $(\Omega, \pi)$, 
$f$ be a positive function on $\Omega$ 
and let $f_{r}$ be the $r$--regularization of $f$. Then
\begin{itemize}

\item $f_{r}$ is $r$--regular;

\item For any $\omega\in \Omega$ with $f_{r}(\omega)>f(\omega)$ there exists a geodesic
path $P$ on the graph $(\Omega,\{(w,w'):\,w\neq w',\,Q(w,w')\neq 0\})$
starting at $\omega$ such that
$$
f_{r}(P[\tau])=\frac{f(P[\length{P}])}{r^{\length{P}-\tau}}
$$
for all $\tau\in[0,\length{P}]$.
\end{itemize}
\end{lemma}

\bigskip

The strategy of proving Theorem~\ref{th: MLSI=reg-MLSI} is straightforward:
for any positive function $f$ on $\Omega$ we consider its $\Upsilon$--regularization $f_{\Upsilon}$
and show that the entropies are related as
${\rm Ent}_\pi(f_{\Upsilon})\geq c\,{\rm Ent}_\pi(f)$, whereas $\Dir_\pi(\log f_{\Upsilon},f_{\Upsilon})\leq C\,\Dir_\pi(\log f,f)$ for some universal constants $c,C>0$.
This immediately implies the required result. The necessary auxiliary statements are verified below.

\begin{lemma}\label{l: 910497219287}
Let $Q$ be a reversible Markov generator on a finite probability space $(\Omega, \pi)$, 
$f$ be a positive function on $\Omega$, and let $\gamma$ and $\Upsilon$ be defined
according to \eqref{eq: m 05623058}. Further, let $f_{\Upsilon}$ be the $\Upsilon$--regularization of $f$. Then
$$\Dir_\pi(\log f_{\Upsilon},f_{\Upsilon})\leq C\,\Dir_\pi(\log f,f),$$
where $C>0$ is a universal constant.
\end{lemma}
\begin{proof}
For brevity, denote
$$
u:=\min\limits_{\omega\neq\omega':\, Q(\omega,\omega')\neq 0}Q(\omega,\omega'),
$$
and
$$
V(\tilde\omega):=\sum_{\omega'':\;\omega''\neq\tilde\omega}Q(\tilde\omega, \omega'')\big(f(\tilde\omega)-f(\omega'')\big)\log\frac{f(\tilde\omega)}{f(\omega'')},
\quad \tilde\omega\in\Omega.
$$
Observe that, in view of the definition of $\Upsilon$, we have for any $\tilde\omega\in\Omega$:
\begin{equation}\label{eq: 02958104982174}
\sum_{\omega:\;\omega\neq\tilde\omega}
\frac{\pi(\omega)}{u\,\Upsilon^{\dist(\tilde \omega,\omega)}}\leq \sum_{\omega:\;\omega\neq\tilde\omega}
\frac{\gamma \pi(\tilde \omega)}{u\,\Upsilon^{\dist(\tilde \omega,\omega)}}\leq \frac{\pi(\tilde \omega)}{8\gamma}.
\end{equation}

Fix any pair of adjacent vertices $\omega,\omega'$ of $G_{\Omega,Q}:=(\Omega,\{(w,w'):\,w\neq w',\,Q(w,w')\neq 0\})$. Without loss of generality, we can assume that
$f(\omega)\geq f(\omega')$. We shall consider three cases.
\begin{itemize}
\item $f(\omega)=f_{\Upsilon}(\omega)\geq f_{\Upsilon}(\omega')$. Since $f_{\Upsilon}(\omega')\geq f(\omega')$, in this case we have
$$
\big(f_{\Upsilon}(\omega)-f_{\Upsilon}(\omega')\big)\log\frac{f_{\Upsilon}(\omega)}{f_{\Upsilon}(\omega')}\leq
\big(f(\omega)-f(\omega')\big)\log\frac{f(\omega)}{f(\omega')}.
$$
\item $f(\omega)<f_{\Upsilon}(\omega)$ and $f_{\Upsilon}(\omega)\geq f_{\Upsilon}(\omega')$.
In this case necessarily there is a vertex $\tilde\omega=\Rel_{f}(\omega)$ with
$$
f_{\Upsilon}(\omega)=\frac{f(\tilde\omega)}{\Upsilon^{\dist(\tilde \omega,\omega)}}=\frac{f_{\Upsilon}(\tilde\omega)}{\Upsilon^{\dist(\tilde \omega,\omega)}};
$$
moreover, there is a vertex $\hat \omega$ adjacent to $\tilde \omega$ and with $\dist(\hat\omega,\omega)=\dist(\tilde \omega,\omega)-1$
such that $f_{\Upsilon}(\hat\omega)=\frac{1}{\Upsilon}\,f(\tilde\omega)$ (see Lemma~\ref{l: 317610498}). Note that 
$f_{\Upsilon}(\omega')\geq \frac{f(\hat\omega)}{\Upsilon^{\dist(\omega',\hat \omega)}}\geq \frac{f(\hat\omega)}{\Upsilon^{\dist(\tilde \omega, \omega)}}$.
Using this, we can write 
\begin{align*}
\big(f_{\Upsilon}(\omega)-f_{\Upsilon}(\omega')\big)\log\frac{f_{\Upsilon}(\omega)}{f_{\Upsilon}(\omega')}
&\leq \frac{1}{\Upsilon^{\dist(\tilde \omega,\omega)}}\big(f(\tilde\omega)-f(\hat\omega)\big)\log\frac{f(\tilde\omega)}{f(\hat\omega)}.
\end{align*}

\item $f_{\Upsilon}(\omega)< f_{\Upsilon}(\omega')$. Similarly to the previous case, there is 
a vertex $\tilde\omega$ with
$$
f_{\Upsilon}(\omega')=\frac{f(\tilde\omega)}{\Upsilon^{\dist(\tilde \omega,\omega')}}=\frac{f_{\Upsilon}(\tilde\omega)}{\Upsilon^{\dist(\tilde \omega,\omega')}},
$$
and there is a vertex $\hat \omega$ adjacent to $\tilde \omega$ and with $\dist(\hat\omega,\omega')=\dist(\tilde \omega,\omega')-1$
such that $f_{\Upsilon}(\hat\omega)=\frac{1}{\Upsilon}\,f(\tilde\omega)$.
Hence,
\begin{align*}
\big(f_{\Upsilon}(\omega)-f_{\Upsilon}(\omega')\big)\log\frac{f_{\Upsilon}(\omega)}{f_{\Upsilon}(\omega')}
&\leq \frac{1}{\Upsilon^{\dist(\tilde \omega,\omega')}}\big(f(\tilde\omega)-f(\hat\omega)\big)\log\frac{f(\tilde\omega)}{f(\hat\omega)}
\end{align*}
\end{itemize}

Summing the above estimates and using the chain reversibility, we obtain
\begin{align*}
&\sum_{\omega'\neq\omega}\pi(\omega) Q(\omega,\omega')\big(f_{\Upsilon}(\omega)-f_{\Upsilon}(\omega')\big)\log\frac{f_{\Upsilon}(\omega)}{f_{\Upsilon}(\omega')}\\
&\hspace{1cm}\leq \sum_{\omega'\neq\omega}\pi(\omega)Q(\omega,\omega')\big(f(\omega)-f(\omega')\big)\log\frac{f(\omega)}{f(\omega')}
+\sum_{\omega'\neq\omega}\;
\sum_{\tilde\omega:\;\tilde\omega\neq \omega}\frac{4\pi(\omega) Q(\omega,\omega')}{u\,\Upsilon^{\dist(\tilde \omega,\omega)}}\,V(\tilde\omega)\\
&\hspace{1cm}\leq 2\Dir_\pi(\log f,f)
+\sum_{\tilde\omega\in\Omega} \;\sum_{\omega:\;\omega\neq\tilde\omega}
\frac{4\pi(\omega)}{u\,\Upsilon^{\dist(\tilde \omega,\omega)}}\,  V(\tilde\omega)\\
&\hspace{1cm}\leq 2 \Dir_\pi(\log f,f) +\frac{1}{\gamma} \Dir_\pi(\log f,f),
\end{align*}
where the last inequality follows after using \eqref{eq: 02958104982174}.
\end{proof}

\bigskip

In order to verify a counterpart comparison inequality for the entropies,
we need the following simple %equivalent characterization of the entropy. 
relaxation of the duality formula of the entropy.
\begin{lemma}\label{l: 391873049817}
Let $(\Omega,\pi)$ be a finite probability space and $f$ be a positive function on $\Omega$. Then
$$
\Ent_\pi(f)\leq 2 \sup\big\{\Exp_{\pi}[f\tilde h] \text{ with } \tilde h:\Omega \to \R \text{ satisfying } \Exp_\pi [e^{\tilde h}]=1\mbox{ and }
\tilde h\geq \log(1/2)\big\}.
$$
\end{lemma}
\begin{proof}
Let $h$ be any function on $\Omega$ with $\Exp_\pi [e^{h}]=1$.
Define $\tilde h$ via the relation
$$
\exp(\tilde h)=\frac{\exp(h)+1}{2}.
$$
Clearly, $\Exp_\pi [e^{\tilde h}]=1$ and $\tilde h\geq \log(1/2)$.
At the same time, it is easy to check that
$$
\tilde h(\omega)\geq h(\omega)/2
$$
for every $\omega\in \Omega$, whence
$$
\Exp_{\pi}[f\tilde h] \geq \frac{1}{2}\Exp_{\pi}[fh]. 
$$
Applying the variational formula for the entropy \cite[Lemma~3.15]{vanHandel-course}, we get the result.
\end{proof}

We are now ready to prove a lemma which, together with Lemma~\ref{l: 910497219287}, yields Theorem~\ref{th: MLSI=reg-MLSI}:
\begin{lemma}
Let $Q$ be a reversible Markov generator on a finite probability space $(\Omega, \pi)$, 
$f$ be a positive function on $\Omega$, and let $\Upsilon$ be defined
according to \eqref{eq: m 05623058}. Further, let $f_{\Upsilon}$ be the $\Upsilon$--regularization of $f$. Then
$$
{\rm Ent}_\pi(f_{\Upsilon})\geq c\,{\rm Ent}_\pi(f),
$$
where $c>0$ is a universal constant.
\end{lemma}
\begin{proof}
In view of Lemma~\ref{l: 391873049817}, we can find a function $\tilde h$ on $\Omega$ with
$\Exp_\pi [e^{\tilde h}]=1$ and $\tilde h\geq \log(1/2)$, such that
$$
\Exp_{\pi}[f\tilde h]\geq \frac{1}{2}\,\Ent_\pi(f).
$$
Denote the domain of $\Rel_{f,\Omega,\Upsilon}$ by $S$:
$$
S:=\big\{\omega\in\Omega:\;f_{\Upsilon}(\omega)>f(\omega)\big\}.
$$
We clearly have 
\begin{align}
\Exp_{\pi}[f_{\Upsilon}\tilde h]
&=\Exp_{\pi}[f\tilde h]+\sum_{\omega\in S}\pi(\omega) \big(f_{\Upsilon}(\omega)-f(\omega)\big)\tilde h(\omega)\nonumber\\
&\geq
\Exp_{\pi}[f\tilde h]-\log(2)\sum_{\omega\in S}\pi(\omega) f_{\Upsilon}(\omega)\nonumber\\
&\geq \frac12 \,\Ent_\pi(f) -\log(2)\sum_{\omega\in S}\pi(\omega) f_{\Upsilon}(\omega).\label{eq1: entropy-reg}
\end{align}
On the other hand, using the definition of $f_{\Upsilon}$, $\Rel_{f,\Omega,\Upsilon}$, and relation~\eqref{eq: 02958104982174}, we get
$$
\sum_{\omega\in S}\pi(\omega)f_{\Upsilon}(\omega)
\leq \sum_{\tilde\omega\in {\rm Im}\,\Rel_{f,\Omega,\Upsilon}}
\sum_{\omega:\;\omega\neq \tilde\omega}\frac{\pi(\omega) f(\tilde\omega)}{\Upsilon^{\dist(\omega,\tilde\omega)}}
\leq \frac{1}{8\gamma}\sum_{\tilde\omega\in {\rm Im}\,\Rel_{f,\Omega,\Upsilon}}\pi(\tilde \omega) f(\tilde\omega).
$$
Pick a subset $T\subset\{\omega:\;f_{\Upsilon}(\omega)>f(\omega)\}$ of cardinality $|{\rm Im}\,\Rel_{f}|$
such that $\Rel_{f}(T)={\rm Im}\,\Rel_{f}$.
Observe that $f(\omega)\leq \frac{1}{\Upsilon}f(\Rel_f(\omega))\leq \frac{1}{16 \gamma}f(\Rel_f(\omega))$ for every $\omega\in T$. 
Using this, we can write 
\begin{align}
\frac{1}{8\gamma}\sum_{\tilde\omega\in {\rm Im}\,\Rel_{f,\Omega,\Upsilon}}\pi(\tilde \omega) f(\tilde\omega)
&=\frac{1}{4\gamma}\sum_{\tilde\omega\in {\rm Im}\,\Rel_{f,\Omega,\Upsilon}}\pi(\tilde \omega) f(\tilde\omega)- 
\frac{1}{8\gamma}\sum_{\omega \in T} \pi(\Rel_f(\omega)) f(\Rel_f(\omega))\nonumber\\
&\leq \frac{1}{4\gamma}\sum_{\tilde\omega\in {\rm Im}\,\Rel_{f,\Omega,\Upsilon}}\pi(\tilde \omega) f(\tilde\omega)- 
2\sum_{\omega \in T} \pi(\omega) f(\omega)\nonumber\\
&\leq \frac12\, \Exp_{\pi}[f h],\label{eq2: entropy-reg}
\end{align}
where 
$$
h(\omega)=\begin{cases}
\frac{1}{2\gamma},&\mbox{if $\omega\in {\rm Im}\,\Rel_{f}$};\\
-4,&\mbox{if $\omega\in T$};\\
0,&\mbox{otherwise.}
\end{cases}
$$
Using that $\frac{\pi({\rm Im}\,\Rel_{f})}{\pi(T)}\leq \gamma$, it is easy to check that $\Exp_\pi [e^h]\leq 1$. Thus, by the variational formula of the entropy \cite[Lemma~3.15]{vanHandel-course}, we deduce that 
$ \Exp_{\pi}[f h]\leq \Ent_\pi(f)$. 
Using this, together with \eqref{eq1: entropy-reg} and \eqref{eq2: entropy-reg}, we finish the proof.
\end{proof}

\section{A comparison technique for MLSI}\label{sec: mlsi-comparison}

The goal of this section is to prove Theorem~\ref{th: comparison-mlsi}. 
A crucial role in comparison techniques for Markov chains is played by the {\it canonical path} method. In its most general setting, we are given a collection of paths $\mathcal P$ on $(\Omega,\pi)$
and a collection of non-negative weights $\mathcal W=(w_{P})_{P\in \mathcal P}$ indexed
over the paths, and would like to bound the weighted sum
$$
\sum_{P\in\mathcal P}w_P\,\big(f(P[\length{P}])-f(P[0])\big)\log\frac{f(P[\length{P}])}{f(P[0])}
$$
from above in terms of $\Dir_\pi(\log f,f)$.
Note that, unlike in the case of squares of differences which are dealt with in the context of the Poincar\'e or log-Sobolev inequalities,
the expression
$$(f(P[\length{P}])-f(P[0]))\log\frac{f(P[\length{P}])}{f(P[0])}$$
does not
split into corresponding quantities for adjacent points in the path, unless some assumptions on $f$ are imposed.
Indeed even in the situation when $\length{P}=2$, the above quantity can be arbitrarily large compared to
$$
(f(P[1])-f(P[0]))\log\frac{f(P[1])}{f(P[0])}+(f(P[2])-f(P[1]))\log\frac{f(P[2])}{f(P[1])}
$$
(for example, taking $f(P[2])=1$, $f(P[0])=\varepsilon$ and $f(P[1])=(\log 1/\varepsilon)^{-1}$,
we clearly get that $(f(P[2])-f(P[0]))\log\frac{f(P[2])}{f(P[0])}=\Theta(\log 1/\varepsilon)$
while $(f(P[1])-f(P[0]))\log\frac{f(P[1])}{f(P[0])}+(f(P[2])-f(P[1]))\log\frac{f(P[2])}{f(P[1])}=\Theta(\log\log 1/\varepsilon)$
when $\varepsilon\to 0$).

However, when the function $f$ is $r$--regular in the sense introduced in this paper, the following simple estimate holds:
\begin{lemma}\label{lem: path-decomp}
Let $S$ be a finite set, $(x_i)_{0\leq i\leq T}$ be a sequence of elements (not necessarily distinct) in $S$ and let $f:\, S\to \R_+$ be a function such that $\max\big(\frac{f(x_i)}{f(x_{i-1})},
\frac{f(x_{i-1})}{f(x_i)}\big)\leq r$, $1\leq i\leq T$, for some $r\geq 2$. %an $r$--regular function for some $r\geq 2$.
Then
$$\big(f(x_T)-f(x_0)\big)\log\frac{f(x_T)}{f(x_0)}\leq C\, \big(1+(T-1)^2\,\log(r)\big)\;
\sum_{t=1}^{T}\big(f(x_t)-f(x_{t-1})\big)\log\frac{f(x_t)}{f(x_{t-1})},
$$
where $C>0$ is a universal constant.
\end{lemma}
\begin{proof}
Note that if $T=1$ then there is nothing to prove so we assume that $T>1$. Without loss of generality, we may also assume that $f(x_T)>f(x_0)$ and that $f(x_T)=\max_{t=0,\ldots, T} f(x_t)$. Indeed, if that was not the case, then we would let $t_0$ be such that $f(x_{t_0})=\max_{t=0,\ldots, T} f(x_t)$, then write 
$$
\big(f(x_T)-f(x_0)\big)\log\frac{f(x_T)}{f(x_0)}\leq \big(f(x_{t_0})-f(x_0)\big)\log\frac{f(x_{t_0})}{f(x_0)},
$$
and work with the truncated sequence $(x_i)_{0\leq i\leq t_0}$.

Define $H$ as the collection of all indices $t\in [T]$ such that
$$
\frac{f(x_t)}{f(x_{t-1})}-1\leq \frac{f(x_T)-f(x_0)}{2\,T\,f(x_T)},
$$
and denote by $H^c$ the complement of $H$ in $[T]$. 
Observe that
\begin{align*}
\sum_{t\in H}\big(f(x_t)-f(x_{t-1})\big)&\leq f(x_T)\,\sum_{t\in H}\bigg(\frac{f(x_t)}{f(x_{t-1})}-1\bigg)\\
&\leq \frac{1}{2}\big(f(x_T)-f(x_0)\big)\\
&=\frac{1}{2}\bigg(\sum_{t\in H}\big(f(x_t)-f(x_{t-1})\big)
+\sum_{t\in H^c}\big(f(x_t)-f(x_{t-1})\big)\bigg),
\end{align*}
whence
$$
f(x_T)-f(x_0)\leq 2\sum_{t\in H^c}\big(f(x_t)-f(x_{t-1})\big).
$$
Therefore, denoting $\delta:=\frac{f(x_T)}{f(x_0)} $, we can write 
\begin{align*}
\big(f(x_T)-f(x_0)\big)\log\frac{f(x_T)}{f(x_0)} 
&\leq 2\, \log(\delta)\,  \sum_{t\in H^c}\big(f(x_t)-f(x_{t-1})\big)\\
&\leq \frac{2\log(\delta)}{\log\big(1+\frac{(1-\delta^{-1})}{2T}\big)}\,  \sum_{t\in H^c}\big(f(x_t)-f(x_{t-1})\big)\log\frac{f(x_t)}{f(x_{t-1})},
\end{align*}
where we used that $\frac{f(x_t)}{f(x_{t-1})}\geq 1+\frac{(1-\delta^{-1})}{2T}$ when $t\in H^c$. 
Now using that $\log x \geq \frac12 (x-1)$ when $1\leq x\leq 2$, we get 
$$
\big(f(x_T)-f(x_0)\big)\log\frac{f(x_T)}{f(x_0)} 
\leq  8T\, \frac{\log(\delta)}{(1-\delta^{-1})}\,  \sum_{t\in H^c}\big(f(x_t)-f(x_{t-1})\big)\log\frac{f(x_t)}{f(x_{t-1})}.
$$
It remains to note that the function $s\to \frac{\log(s)}{(1-s^{-1})}$, $s\geq 1$, is increasing in $s$
and use that $\delta\leq r^{T}$ to finish the proof. 
\end{proof}

With this lemma in hand, the proof of Theorem~\ref{th: comparison-mlsi} will easily follow. 
\begin{proof}[Proof of Theorem~\ref{th: comparison-mlsi}]
Fix $r\geq 2$, $f \in \Reg(Q,r)$, and let $\weight$ be a $(Q, \tilde Q)$-flow.
First recall the following characterization of entropy (see \cite[Problem~3.13a]{vanHandel-course}),
\begin{equation}\label{eq: caract-entropy}
\begin{split}
\Ent_\pi(f)&= \inf_{t>0} \Exp_\pi\big[ f\log f- f\log t-f+t\big]\\
&=\inf_{t>0}\sum\limits_{\omega\in\Omega}\big(f(\omega)\log f(\omega)- f(\omega)\log t-f(\omega)+t\big)\pi(\omega)
\end{split}
\end{equation}
(with the corresponding formula for $\Ent_{\tilde\pi}(f)$),
and note that $x\log x-x\log y-x+y\geq 0$ for any $x,y>0$. The term-wise
comparison and the assumption $\pi(\omega)\leq a\,\tilde\pi(\omega)$ then yields $\Ent_\pi(f)\leq a\, \Ent_{\tilde \pi}(f)$. 
It remains to compare the two Dirichlet forms $\Dir_\pi$ and $\tilde \Dir_{\tilde \pi}$ associated with $(\pi, Q)$ and $(\tilde \pi, \tilde Q)$ respectively. 
To this aim, we write 
\begin{align*}
\tilde \Dir_{\tilde \pi}(f, \log f)&= \frac12\sum_{\omega, \omega'\in \Omega} \tilde \pi(\omega) \tilde Q(\omega, \omega') \big(f(\omega)-f(\omega')\big) \log \frac{f(\omega)}{f(\omega')}\\
&=\frac12  \sum_{\omega, \omega'\in \Omega}\sum_{P\in \path_{\omega,\omega'}} \weight(P)   \big(f(P[T])-f(P[0])\big) \log \frac{f(P[T])}{f(P[0])},
\end{align*}
where we denoted by $T=T(P)$ the length of a path $P$. 
Applying Lemma~\ref{lem: path-decomp}, we get for some universal constant $C$ that
\begin{align*}
\tilde \Dir_{\tilde \pi}(f, \log f)&\leq C\sum_{\omega, \omega'\in \Omega}\sum_{P\in \path_{\omega,\omega'}} \weight(P)  \big(1+(T-1)^2\,\log(r)\big)\;
\sum_{t=1}^{T}\big(f(x_t)-f(x_{t-1})\big)\log\frac{f(x_t)}{f(x_{t-1})}\\
&\leq 2C\,A(\weight, r)\, \Dir_\pi(f, \log f),
\end{align*}
where $A(\weight, r)$ is given by \eqref{nfaofjnpifqwnpijnfijnw}.
Putting together the above estimates, we finish the proof. 
\end{proof}

\section{MLSI for the switch chain}\label{sec: mlsi-switch}
\subsection{Preliminaries}
In this section, we establish an optimal Modified log-Sobolev Inequality
for the switch chain on regular bipartite graphs. 
We start by considering a lower bound for the MLSI constant.
\begin{prop}[Lower bound for the MLSI constant]
Let $2\leq d\leq \frac{n}{2}$. 
The modified log-Sobolev constant of $(\BipGSet_n(d), \Psimple, Q_u)$ is at least $c nd$, for some universal constant $c>0$. 
\end{prop}
\begin{proof}
Denote by $\alpha$ the optimal MLSI constant, so that 
$$
\alpha=\sup \frac{\Ent_{\Psimple}(f)}{\Dir_{\Psimple}(f,\log f)},
$$
where the supremum is taken over all functions $f:\, \BipGSet_n(d)\to \R_+$. 
To obtain the required lower bound on $\alpha$, we shall use a test function.  
Define $f:\, \BipGSet_n(d)\to \R_+$ as 
$$
f(G)=\begin{cases}
2,& \text{ if $(1,1)$ is an edge in $G$};\\
1,& otherwise. 
\end{cases}
$$
Note that given $G\in  \BipGSet_n(d)$ with $(1,1)$ {\it not} as an edge, there are at most $d^2$ adjacent graphs $G'$ for $G$ having the edge $(1,1)$. 
Using this, we can write 
\begin{align*}
\Dir_{\Psimple}(f,\log f)&=\frac{\log 2}{2}  \sum_{\underset{f(G)=1}{G\in\BipGSet_n(d)}} \sum_{\underset{G'\sim G,\, f(G')=2}{G'\in  \BipGSet_n(d)}} \Psimple(G)Q_u(G,G')\\
&=\frac{\log 2}{nd(nd-1)}  \sum_{\underset{f(G)=1}{G\in\BipGSet_n(d)}}\Psimple(G)\,\big|\big\{G'\sim G:\; f(G')=2\big\}\big|
\leq \frac{2\log 2}{n^2}.
\end{align*}
On the other hand,
it follows from $d$-regularity that 
$$
\vert\{G\in \BipGSet_n(d):\, f(G)=2\}\vert =\frac{d}{n}\vert \BipGSet_n(d)\vert,
$$
whence $\Exp_{\Psimple}f=\frac{n+d}{n}$, and
$$
\Ent_{\Psimple}(f)
=\frac{n-d}{n}\log \frac{n}{n+d}+\frac{2d}{n}\log \frac{2n}{n+d}
\geq \frac{c'd}{n}
$$
for some universal constant $c'>0$.

%$
% \Ent_{\Psimple}(f)\geq  \frac{d}{n}. 
%$

Putting these estimates together, we deduce that 
$$
\alpha\geq \frac{c'nd}{2\log 2}, 
$$
and finish the proof.
\end{proof}

We further note that the lower bound $\Omega(nd)$ for the MLSI constant
can be obtained indirectly, by bounding the mixing time of the switch chain by $\Omega(nd\log n)$, and applying relation \eqref{eq: mix-mlsi}.
We include this alternative argument, which also shows that our mixing time upper bound is sharp, for completeness.
\begin{prop}[A lower bound for the mixing time]\label{prop:lower bound}
There are universal constants $C,c>0$ with the following property. Let $n\geq C$ and let $2\leq d\leq \sqrt[4]{n}$. Then the total variation mixing time $\tmix(Q_u,\frac{1}{4})$ of the switch chain $(\BipGSet_n(d), \Psimple, Q_u)$ is bounded below by $c nd\log n$. 
\end{prop}
\begin{proof}
We will assume in the proof that $n$ is sufficiently large. In order to derive a lower bound on the mixing time, we will make use of a distinguishing statistic. 
We start our chain with a graph $G_0$ which contains all edges of the form $(i,i)$, $1\leq i\leq n$.
Let $T$ be a positive integer parameter and let $G_1,\dots,G_T$ be the steps
of the switch chain starting at $G_0$. For each $0\leq t\leq T$, denote by $(\xi_{i,t})_{1\leq i\leq n}$ the Bernoulli variables indicating the ``diagonal'' edges in the graph $G_t$, i.e $\xi_{i,t}=1$ whenever the edge $(i,i)$ is present in $G_t$. 
Given $0\leq t\leq T-1$ and $1\leq i\leq n$, note that conditioned on a realization of $G_t$ with $\xi_{i,t}=1$, we have $\xi_{i,t+1}=1$ with conditional probability at least $1-\frac{2}{nd}$. Thus, we have 
$$
\Exp[ \xi_{i,t+1}]\geq (1-\frac{2}{nd}) \Exp[ \xi_{i,t}].
$$
Iterating this inequality, we get 
$$
\Exp[ \xi_{i,T}]\geq (1-\frac{2}{nd})^T,
$$
for every $i=1,\ldots,n$. 

Now given $1\leq i\neq j\leq n$, and  conditioned on $(\xi_{i,t}, \xi_{j,t})$, it is not difficult to chech that  $\xi_{i,t+1}\xi_{j,t+1}=1$ with conditional probability at most 
$$
\begin{cases}
\frac{c}{n^2} &\mbox{ if\,  $\xi_{i,t} \xi_{j,t}=0$}\\
&\\
1-\frac{4}{nd}+ \frac{c}{n^2}&\mbox{ if\,  $\xi_{i,t} \xi_{j,t}=1$},
\end{cases}
$$
for some universal constant $c$. 
Putting these estimates together, we get after iteration that 
$$
\Exp[\xi_{i,T}\xi_{j,T}]\leq (1-\frac{4}{nd})^T+ \frac{c'd}{n},
$$
for some universal constant $c'$. 

We deduce from the above that for every $1\leq i\neq j\leq n$ we have
$$
\cov{(\xi_{i,T},\xi_{j,T})}\leq \frac{c'd}{n}.
$$
Denoting $D_T= \sum_{i=1}^n \xi_{i,T}$, we deduce from the above that 
$$
\Exp[D_T] \geq n(1-\frac{2}{nd})^T \quad \text{and}\quad \Var(D_T)\leq Cnd,
$$
for some universal constant $C$. 

On the other hand, the $d$--regularity immediately implies that
the expected number of ``diagonal'' edges in a uniform random graph on $\BipGSet_n(d)$
is $d$ while the variance of that number is at most $nd$. 
It remains to apply \cite[Proposition~7.9]{book-peres} to finish the proof. 
\end{proof}

Following the approach from \cite{TY2020+}, we develop a comparison procedure between $(\BipGSet_n(d), \Psimple, Q_u)$ and the switch chain on multigraphs generated according to the configuration model,
which in turn can be compared to the random transposition chain on the set of permutations. 
We denote by $\ConfBipGSet_n(d)$ the set of all $d$--regular bipartite 
multigraphs on $[n^{(\ell)}]\sqcup[n^{(r)}]$ and equip it with 
the probability measure $\Pconfig$ induced by the configuration model. 
Recall that for any $G\in \ConfBipGSet_n(d)$ 
$$
\Pconfig (G)= \frac{(d!)^{2n}}{(nd)! \prod_{1\leq i,j\leq n} \mult_{G}(i,j)!},
$$
where $ \mult_{G}(i,j)$ denotes the multiplicity of the edge $(i,j)$ in $G$. 
When $n$ is large enough, we have the following estimate (see \cite[Theorem~6.2]{Janson})
\begin{equation}\label{eq: size-multigraph}
\frac12 e^{-\frac{(d-1)^2}{2}}\leq \Pconfig\big(\BipGSet_n(d)\big) \leq 2e^{-\frac{(d-1)^2}{2}}. 
\end{equation}
The generator $Q_c$ of the switch chain on $\ConfBipGSet_n(d)$ is defined 
for any $G_1, G_2\in \ConfBipGSet_n(d)$ by
$$
Q_c(G_1,G_2):=\begin{cases}
 \frac{\mult_{G_1}(i,j)\,\mult_{G_1}(i',j') }{nd(nd-1)/2},&\mbox{if $G_2\in \neigh(G_1)$ is obtained from $G_1$}\\
 &\mbox{ by the switching $\langle i,i',j,j'\rangle$};\\
-\sum_{G'\in \neigh(G_1)} Q_c(G_1,G'),&\mbox{if $G_1=G_2$};\\
&\\
0,&\mbox{otherwise.}
\end{cases}
$$
In the above notation, $\langle i,i',j,j'\rangle$ denotes the switching destroying the edges $(i,j)$ and $(i',j')$, and replacing them with $(i,j')$ and $(i',j)$.

It is known that the random transposition chain on the set of permutations of $[nd]$ satisfies the Modified log-Sobolev Inequality with constant $nd$ \cite{G}. 
Exploiting the intimate relation between the random transposition chain and the switch chain on the configuration model, it is easy to derive the 
following Modified log-Sobolev Inequality for the latter (see \cite[Proposition~2.2]{TY2020+} for details). 
\begin{prop}\label{eq: mlsi-conf}
For any $2\leq d\leq n/2$,  $(\ConfBipGSet_n(d), \Pconfig, Q_c)$ satisfies the Modified log-Sobolev Inequality with constant $c nd$ for some universal constant $c>0$.
\end{prop}

\subsection{Simple paths and $s$-neighborhoods}\label{ajfnpakfjnwpfiqjwnfp}
For the remainder of the paper, fix $2\leq d\leq n/2$. 
We will use the last proposition to construct an auxiliary Markov chain on $\BipGSet_n(d)$ satisfying the Modified log-Sobolev Inequality with a constant of order $O(nd)$, and then use this auxiliary chain 
with the comparison Theorem~\ref{th: comparison-mlsi}. In order to verify that the auxiliary chain does satisfy the MLSI with a satisfactory constant, we will construct 
for any given positive function $f$ on $\BipGSet_n(d)$ an appropriate extension $\tilde f$ to the set $\ConfBipGSet_n(d)$.
Following \cite{TY2020+}, we interpret $\BipGSet_n(d)$ as a boundary for $\ConfBipGSet_n(d)\setminus \BipGSet_n(d)$ and define $\tilde f$ as 
a ``relative'' of the standard harmonic extension of $f$.
While the harmonic extension is constructed by launching a random walk from the given point in $\ConfBipGSet_n(d)\setminus \BipGSet_n(d)$
and averaging the values of $f$ where it hits the boundary, the strategy developed in \cite{TY2020+} is to construct specific ``direct'' paths to the boundary $\BipGSet_n(d)$
which would make the result of the averaging tractable.
As these paths crucially depend on some properties of the corresponding starting multigraph, 
let us start by partitioning $\ConfBipGSet_n(d)$.

\begin{defi}[A partition of $\ConfBipGSet_n(d)$, \cite{TY2020+}]\label{def: category}
Let $\cconst:=\lfloor \log\log n\rfloor$. We write 
$$
\ConfBipGSet_n(d)=  \bigsqcup_{k=0}^{\cconst} \categ(k)\sqcup \Ugly(\cconst),
$$
where $\Ugly(\cconst):=  \categ([0,\cconst])^c$, $\categ([0,\cconst]):= \bigsqcup_{k=0}^{\cconst} \categ(k)$, and $\categ(k)$ is defined as 
the set of multigraphs $G\in \ConfBipGSet_n(d)$ which satisfy all of the following: 
\begin{itemize}
\item $G$ has exactly $k$ multiedges of multiplicity $2$; 
\item None of those multiedges are incident to one another; 
\item $G$ has no multiedges of multiplicity three or greater. 
\end{itemize}
\end{defi}
Note that with this definition, we have $\BipGSet_n(d)=\categ(0)$. 
Multigraphs in $\categ([1,\cconst])$ have a simple structure allowing to build ``direct'' paths from them to the boundary. The paths are formed by the 
%Indeed, the above properties ensure that it is possible to construct
simple switchings which destroy the multiedges one at a time while not ''interfering`` with one and another.
Thus, for every $G'\in \categ(k)$, we construct a unique family of paths from $G'$ to $\BipGSet_n(d)$ of length $k$ where each step of the path destroys a multiple edge.
Such paths will be called ``simple paths'', and are formally defined as follows.
\begin{defi}[Simple paths, \cite{TY2020+}]
Given $1\leq k\leq \cconst$ and $G'\in \categ(k)$, denote by $\{(i_s,j_s)\}_{1\leq s\leq k}$ the multiedges of $G'$ of multiplicity $2$
arranged in increasing order of $(i_s)_{1\leq s\leq k}$. 
A {\it simple path} $P$ starting at $G'$ is a path of length $k$ where $P[t+1]$
is obtained from $P[t]$ via the simple switching $\langle i_{t+1},i_{t+1}',j_{t+1},j_{t+1}'\rangle$, such that $i_{t+1}'\in [n^{(\ell)}], j_{t+1}'\in [n^{(r)}]$ satisfy all of the following conditions:
\begin{itemize}
\item For every $0\leq t<k$, we have 
$$
i_{t+1}'\not\in \{i_s\}_{1\leq s\leq k}, \quad j_{t+1}'\not\in \{j_s\}_{1\leq s\leq k},
$$
and all $(i_s')_{1\leq s\leq k}$ (resp. $(j_s')_{1\leq s\leq k}$) are pairwise distinct. 
\item For every $0\leq t<k$, $\mult_{G'}(i_{t+1}',j_{t+1})= \mult_{G'}(i_{t+1},j_{t+1}')=0$. 
\end{itemize}
\end{defi}

It can be verified that (with our choice of $\cconst$) simple paths exist for every $G'\in \categ(k)$, and each simple path is uniquely determined by its starting point and endpoint (see \cite[Section~3]{TY2020+} for details). 
Note that the endpoint of a simple path belongs to $\BipGSet_n(d)$. 
\begin{defi}[$s$--neighborhood, \cite{TY2020+}]
The set of all endpoints of simple paths starting at $G'\in \ConfBipGSet_n(d)\setminus \BipGSet_n(d)$ will be denoted by $\SNeigh(G')$ and called {\it the $s$--neighborhood}
of the graph.
\end{defi}
%Additionally, \co{denote by $\widetilde\SNeigh(G')$ the $s$--neighborhood of $G'$ whenever $G'\in\ConfBipGSet_n(d)\setminus \BipGSet_n(d)$ and set $\widetilde\SNeigh(G'):=\{G'\}$ whenever $G'\in \BipGSet_n(d)$.}
Additionally, when $G'\in \BipGSet_n(d)$, we set $\SNeigh(G'):=\{G'\}$.
When $G'\in \categ(k)$, $1\leq k\leq \cconst$, the $s$--neighborhood of $G'$ satisfies (see \cite[Section~3]{TY2020+}):
\begin{equation}\label{eq: s-neighborhood}
\vert\SNeigh(G)\vert \in \big[\frac{(nd)^k}{2}, (nd)^k\big].
\end{equation}
We will also need to control the number of $s$--neighborhoods which contain a given simple graph: for any given $G\in \BipGSet_n(d)$ we have \cite[Section~3]{TY2020+}
\begin{equation}\label{eq: number-neighborhood}
\big|\big\{G'\in \categ(k):\; G\in \SNeigh(G')\big\}\big| \leq  \frac{(nd)^{k}}{k!}(d-1)^{2k},
\quad 1\leq k\leq \cconst.
\end{equation}
In the sequel, we will use the notation $\tuple=(G_1,G_1',G_2,G_2')$
for any $4$--tuple of graphs such that $G_1'\sim G_2'$ are in $\ConfBipGSet_n(d)$, $G_1\in \SNeigh(G_1')$, and 
$G_2\in \SNeigh(G_2')$.

\subsection{Matchings and connections}\label{sec: paths}

In order to make use of the comparison Theorem~\ref{th: comparison-mlsi}, we will need to construct a special family of paths between elements 
of $\BipGSet_n(d)$. As our auxiliary Markov chain on $\BipGSet_n(d)$
will be ``inherited'' from the switch chain on $\ConfBipGSet_n(d)$,
that family of paths will be determined by the structure of 
$\ConfBipGSet_n(d)$.
%This technical task was extensively elaborated in \cite{TY2020+} and we only recall here the necessary ingredients. 
Again, we shall rely on the constructions from \cite{TY2020+}. 
It was observed in \cite{TY2020+} that for a large proportion of adjacent multigraphs in $\ConfBipGSet_n(d)$
there is a natural bijective mapping between their respective $s$--neighborhoods.
\begin{defi}[Perfect pairs, \cite{TY2020+}]
Given $1\leq k\leq \cconst$, a pair of adjacent graphs $(G_1,G_2)\in \categ(k)\times \categ(k)$ is referred to as a perfect pair if the switching $\langle i,i',j,j'\rangle$ used to obtain $G_2$ from $G_1$ satisfies the following conditions:
\begin{itemize}
\item Vertices $i,i',j,j'$ are not incident to any multiedges.
\item Vertices $i,i',j,j'$ are not adjacent to vertices incident to some multiedges.
\end{itemize}
Denote by $\Perfmatch(k)$ the set of perfect pairs 
from $\categ(k)\times \categ(k)$ and set $\Perfmatch:= \bigsqcup_{k=1}^{\cconst} \Perfmatch(k)$. 
\end{defi}
 
\begin{prop}[Matchings, {\cite[Section~5]{TY2020+}}]\label{prop: matching}
Let $G_1'\in \categ([1,\cconst])$. Then the following assertions hold.
\begin{itemize}
\item Let $G_2'\in \categ([1,\cconst])$ be such that $(G_1',G_2')\in \Perfmatch$. Then there is a bijective mapping $\psi_{G_1',G_2'}: \SNeigh(G_1') \to \SNeigh(G_2')$
such that $\psi_{G_2',G_1'}(\psi_{G_1',G_2'}(G))=G$ for all $G\in \SNeigh(G_1')$ and
$$\mbox{$G$ is adjacent to $\psi_{G_1',G_2'}(G)$ for all $G\in \SNeigh(G_1')$}.$$
\item Let $G_1\in \SNeigh(G_1')$ and let $G_2\in \BipGSet_n(d)$ be adjacent to $G_1$. 
Then there exists at most one multigraph $G_2'\in  \categ([1,\cconst])$ such that $(G_1',G_2')\in \Perfmatch$ and $\psi_{G_1',G_2'}(G_1)=G_2\in  \SNeigh(G_2')$.
\end{itemize}
\end{prop}

The above proposition will play a crucial role in our comparison procedure as it associates a family of paths of minimal length $1$ for most pairs of adjacent multigraphs. When $\tuple=(G_1,G_1',G_2,G_2')$ is such that $(G_1',G_2')\in \Perfmatch$, $G_1\in \SNeigh(G_1')$, and $G_2=\psi_{G_1',G_2'}(G_1)$, we let $\path_\tau:= (G_1,\psi_{G_1',G_2'}(G_1)=G_2)$ to be the path of length one from $G_1$ to $G_2$. 

For the pairs of multigraphs which are not perfect, a different construction is required. 
Let us define 
\begin{align*}
\imperfect(\cconst):=\big\{ (G_1', G_2'):\,   &G_1'\sim G_2',\, G_1', G_2'\in \categ([0,\cconst]),\\
&(G_1', G_2')\notin \BipGSet_n(d)\times \BipGSet_n(d)\big\} \setminus \big(\bigsqcup_{k=1}^\cconst \Perfmatch(k)\big),
\end{align*}
and consider 
$$
\tupleset:=\big\{ (G_1,G_1',G_2,G_2'):\; (G_1',G_2')\in \imperfect(\cconst),\; G_1\in\SNeigh(G_1'),\;G_2\in\SNeigh(G_2') 
\big\}.
$$
We recall that in a tuple $\tuple=(G_1,G_1',G_2,G_2')$, we may possibly have that $G_1=G_1'$ or $G_2=G_2'$.
One of the technical contributions in \cite{TY2020+} consists in constructing for every given tuple $\tuple=(G_1,G_1',G_2,G_2')\in \tupleset$ a path $\path_\tuple$ in $\BipGSet_n(d)$ 
starting at $G_1$ and ending at $G_2$ (called a {\it connection}) having a set of special properties
useful in the context of functional inequalities on $\BipGSet_n(d)$. 
The definition of a connection is very technical (see \cite[Definition~6.11]{TY2020+})
and we prefer not to include it in this paper. Rather, we provide a proposition which establishes
existence of certain paths satisfying properties crucial to us (we refer the reader to \cite[Sections~6-7]{TY2020+} for a comprehensive treatment):

\begin{prop}[Connections, \cite{TY2020+}]\label{prop: connections}
Assuming $n$ is sufficiently large, there exists a collection of paths in $\BipGSet_n(d)$, $(\path_\tuple)_{\tuple \in \tupleset}$, indexed by $\tupleset$ and satisfying the following conditions.
\begin{itemize}
\item[i.] For every $0\leq k_1,k_2\leq \cconst$ and every $\tuple=(G_1,G_1',G_2,G_2') \in \tupleset$ with $G_1'\in \categ(k_1), G_2'\in \categ(k_2)$, the path $\path_\tuple$ starts at $G_1$, ends at $G_2$, and is of length at most $C (k_1+k_2)$ for some universal constant $C>0$. 
\item[ii.] Given any adjacent graphs $H\sim H'$ in $\BipGSet_n(d)$, we have 
$$
\sum_{\underset{(H,H')\in \path_\tuple}{\tuple =(G_1,G_1',G_2,G_2') \in \tupleset}} \frac{\Pconfig(G_1') Q_c(G_1',G_2')}{\vert \SNeigh(G_1')\vert\, \vert \SNeigh(G_2')\vert} \length{\path_\tuple}^2\leq C\,\frac{\Psimple(H) Q_u(H,H')}{\sqrt{n}},
$$
for some universal constant $C>0$. 
\end{itemize}
\end{prop}
\begin{proof}
We define $(\path_\tuple)_{\tuple \in \tupleset}$ as the set of {\it connections}
\cite[Definition~6.11]{TY2020+}. 
The first assertion of the proposition can be deduced from 
%\cite[Lemmas~6.9 \& 6.10]{TY2020+}. 
\cite[Remark~6.13]{TY2020+}. 
For the second assertion, denote 
$$
\gamma:= \sum_{\underset{(H,H')\in \path_\tuple}{\tuple =(G_1,G_1',G_2,G_2') \in \tupleset}} \frac{\Pconfig(G_1') Q_c(G_1',G_2')}{\vert \SNeigh(G_1')\vert\, \vert \SNeigh(G_2')\vert} \length{\path_\tuple}^2.
$$
In view of the first part of the proposition, the definitions of $(\Pconfig,Q_c)$ and $(\Psimple,Q_u)$, and using \eqref{eq: s-neighborhood}, we get 
\begin{align*}
\gamma&\leq C_1\cconst^2
\sum_{0\leq k_1,k_2\leq \cconst}\sum_{\underset{\underset{(H,H')\in \path_\tuple}{(G_1',G_2')\in \categ(k_1)\times \categ(k_2)}}{\tuple =(G_1,G_1',G_2,G_2') \in \tupleset}}\frac{\Pconfig(G_1') Q_c(G_1',G_2')}{(nd)^{k_1+k_2}}\\
&\leq C_2 \cconst^2\Pconfig\big(\BipGSet_n(d)\big)\Psimple(H) Q_u(H,H')  \sum_{0\leq k_1,k_2\leq \cconst}\sum_{\underset{\underset{(H,H')\in \path_\tuple}{(G_1',G_2')\in \categ(k_1)\times \categ(k_2)}}{\tuple =(G_1,G_1',G_2,G_2') \in \tupleset}} \frac{1}{2^{k_1}(nd)^{k_1+k_2}},
\end{align*}
where $C_1,C_2>0$ are universal constants.
By combining \cite[Proposition~6.24]{TY2020+}, \cite[Proposition~6.25]{TY2020+} (while bounding the parameter $r$ there by $2\cconst$) and \cite[Proposition~6.26]{TY2020+}, we get
\begin{align*}
\gamma
\leq \frac{(C_d \cconst)^{C'\cconst}}{n} \Pconfig\big(\BipGSet_n(d)\big)\Psimple(H) Q_u(H,H'),
\end{align*}
where the constant $C_d$ depends only on $d$ and $C'>0$ is a universal constant. 
By the choice of $\cconst$, we get the result provided $n$ is large enough. 
\end{proof}

\subsection{The function extension and the auxiliary chain}\label{aljkdnfpajfnpweifjnwi}
Now, we define the aforementioned auxiliary chain to compare it with the switch chain on the configuration model. 
Now, %\co{given an admissible $4$--tuple $(G_1,G_1',G_2,G_2')$, with
%$G_1',G_2'\in\categ([0,\cconst])$, we set}
given
$G_1'\sim G_2'$ in $\categ([0,\cconst])$ and $G_1\in  \SNeigh(G_1')$, $G_2\in  \SNeigh(G_2')$, set 
$$
\beta_{G_1',G_2'}(G_1,G_2):=\begin{cases}
\frac{\mathbf{1}_{\{G_ 2=\psi_{G_1',G_2'}(G_1)\}}}{\vert \SNeigh(G_1')\vert}&\mbox{if $G_1',G_2'\in\categ([1,\cconst])$ and $(G_1', G_2')\in \Perfmatch$,}\\
\frac{1}{\vert \SNeigh(G_1')\vert\, \vert \SNeigh(G_2')\vert}&
\mbox{otherwise}.
%\mbox{if
%$G_1',G_2'\in\categ([1,\cconst])$ and $(G_1', G_2')\notin \Perfmatch$},\\
%\frac{1}{\vert \SNeigh(G_1')\vert}&\mbox{if
%$G_2=G_2'\in \BipGSet_n(d)$},\\
%\frac{1}{\vert \SNeigh(G_2')\vert}&\mbox{if
%$G_1=G_1'\in \BipGSet_n(d)$.}
\end{cases}
$$
Note that for any pair $G_1'\sim G_2'$ in $\categ([0,\cconst])$,
$$
\sum_{\substack{G_1\in \SNeigh(G_1'),
\\G_2\in \SNeigh(G_2')}}
\beta_{G_1',G_2'}(G_1,G_2) =1,
$$
and that for any $4$--tuple $(G_1,G_1',G_2,G_2')$, $\beta_{G_1',G_2'}(G_1,G_2)=\beta_{G_2',G_1'}(G_2,G_1)$.
We define a Markov generator $\tilde Q_u$ on $(\BipGSet_n(d), \Psimple)$ by setting
for every $G_1\neq G_2$ in $\BipGSet_n(d)$,
\begin{equation}\label{eq: def-aux-chain}
\tilde Q_u(G_1,G_2)
:= \frac{1}{4\Psimple(G_1)} \sum_{
\substack{G_1',G_2'\in \categ([0,\cconst]):\\G_1'\sim G_2',\\G_1\in \SNeigh(G_1'),
G_2\in \SNeigh(G_2')}
} 
\Pconfig(G_1')Q_c(G_1',G_2')\beta_{G_1',G_2'}(G_1,G_2),
\end{equation}
and taking $\tilde Q_u(G_1,G_1):=-\sum_{G_2:\,G_2\neq G_1} \tilde Q_u(G_1,G_2)$. 
Note that $\tilde Q_u(G_1,G_2)=\tilde Q_u(G_2,G_1)$ for all $G_1,G_2$.
Further, for every $G_1\in \BipGSet_n(d)$, we get in view of
\eqref{eq: s-neighborhood}, \eqref{eq: number-neighborhood}, and \eqref{eq: size-multigraph}
\begin{align*}
&\sum\limits_{G_2:\,G_2\neq G_1}\tilde Q_u(G_1,G_2)\\
&\hspace{1cm}\leq \frac{1}{4\Psimple(G_1)} \sum_{\substack{G_1',G_2'\in \categ([0,\cconst]):\\G_1'\sim G_2',\\G_1\in \SNeigh(G_1')}}\Pconfig(G_1')Q_c(G_1',G_2')
\sum_{\substack{G_2\in \SNeigh(G_2')}}
\beta_{G_1',G_2'}(G_1,G_2)\\
%&\hspace{1cm}\leq \frac{1}{2\Psimple(G_1)}\sum_{\substack{G_1',G_2'\in \categ([0,\cconst]):\\G_1'\sim G_2'\\(G_1,G_1',G_2,G_2')\\\mbox{is admissible for some $G_2$}\\{}}}\frac{\Pconfig(G_1')Q_c(G_1',G_2')}{
%|\widetilde\SNeigh(G_1')|}\\
&\hspace{1cm}\leq \frac{1}{4\Psimple(G_1)}
\sum_{\substack{G_1'\in \categ([0,\cconst]):\\G_1\in \SNeigh(G_1')}}\frac{\Pconfig(G_1')}{
|\SNeigh(G_1')|}\leq \sum_{k=0}^{\cconst}
\frac{(d-1)^{2k}}{k!}\frac{e^{-\frac{(d-1)^2}{2}}}{2^k}
\leq 1.
\end{align*}
Thus, the generator $\tilde Q_u$ is 
well defined and is reversible with respect to $\Psimple$. 
Next, we prove that the above auxiliary chain satisfies the Modified log-Sobolev Inequality with constant
of order $O_d(n)$.
For the rest of the subsection, we denote by $\Phi:\, \R_+^2\to\R$ the function defined by $\Phi(x,y)=(x-y)\log\frac{x}{y}$ (note that the function is convex in two variables). 
We first need the following lemma.
\begin{lemma}\label{nfapfienfpiwjfnpqifejnpqifjn}
Let $(\Omega,Q,\pi)$ be a reversible Markov chain, and let $f$ be a positive function
on $\Omega$, with $\Exp_\pi\, f=1$ and $f(\omega)\geq \delta$ for all $\omega\in\Omega$
and some parameter $\delta\in(0,1/2]$. Then
$$
\sum_{\omega\in\Omega}\pi(\omega)\,(f(\omega)-1)\log f(\omega)
\leq C'\,|\log \delta|\,\Ent_\pi\, f
$$
for a universal constant $C'>0$.
\end{lemma}
\begin{proof}
We write
$$
\Ent_\pi \,f=\sum_{\omega\in\Omega}\big(1-f(\omega)+f(\omega)\log f(\omega)\big)\pi(\omega),
$$
where $1-f(\omega)+f(\omega)\log f(\omega)\geq 0$ for all $\omega$,
and compare the terms $1-f(\omega)+f(\omega)\log f(\omega)$ and $(f(\omega)-1)\log f(\omega)$.
We consider several cases.
\begin{itemize}
\item $f(\omega)\in[1/2,10]$. We have
$1-f(\omega)+f(\omega)\log f(\omega)\geq \frac16 (f(\omega)-1)^2$
%for a universal constant $c_1>0$ 
while at the same time
 $(f(\omega)-1)\log f(\omega)
\leq 2(f(\omega)-1)^2$. Thus, in this regime we have
$$
1-f(\omega)+f(\omega)\log f(\omega)\geq \frac{1}{12}\,(f(\omega)-1)\log f(\omega).
$$
\item $f(\omega)>10$.
Then
\begin{align*}
&1-f(\omega)+f(\omega)\log f(\omega)\geq \frac12 f(\omega)\log f(\omega),\\
&(f(\omega)-1)\log f(\omega)\leq f(\omega)\log f(\omega),
\end{align*}
implying that 
$$
1-f(\omega)+f(\omega)\log f(\omega)\geq \frac12(f(\omega)-1)\log f(\omega).
$$
%for some universal constant $c_2>0$.
\item $f(\omega)<1/2$. In this range we have
$
(f(\omega)-1)\log f(\omega)\leq -\log f(\omega)
$
whereas
$
1-f(\omega)+f(\omega)\log f(\omega)\geq \frac18$. In view of the assumptions on $f$, this implies
$$
1-f(\omega)+f(\omega)\log f(\omega)\geq \frac{1}{8|\log \delta|}(f(\omega)-1)\log f(\omega).
$$
%for a universal constant $c_3>0$.
\end{itemize}
Combining the above estimates, we get the result.
\end{proof}

\begin{prop}[The MLSI for the auxiliary chain]\label{prop: auxiliary chain}
For every fixed $d\geq 2$ there are $n_d,C_d>0$ depending only on $d$ such that, assuming $n\geq n_d$, $(\BipGSet_n(d), \Psimple, \tilde Q_u)$ satisfies the Modified log-Sobolev Inequality with constant $C_d n$. 
\end{prop}
\begin{proof}
We will deduce the result by an appropriate comparison with the switch chain on the configuration model. 
We shall verify that
for any positive function $f:\, \BipGSet_n(d)\to \R_+$,
$$
\Ent_{\Psimple}(f)\leq  c_dn\,\Dir_{\Psimple}(f,\log f),
$$
for some appropriate constant $c_d$. 
Note that in view of Lemma~\ref{akjfnqoifnfoqifalkdjfn}, we can assume without loss of generality that
$\Exp_{\Psimple}(f)=1$ and
$f(G)\geq c$
for all $G\in \BipGSet_n(d)$, for some universal constant $c>0$.
Using the characterization of entropy in \eqref{eq: caract-entropy}, we get that for any extension $\tilde f$ of $f$ to $\ConfBipGSet_n(d)$,
$$
\Ent_{\Psimple}(f)\leq \max_{G\in\BipGSet_n(d)} \frac{\Psimple(G)}{\Pconfig(G)}  \Ent_{\Pconfig}(\tilde f)\leq C_d \Ent_{\Pconfig}(\tilde f),
$$
where the last inequality follows from \eqref{eq: size-multigraph}, and where $C_d>0$ is a constant depending only on $d$. 
Using Proposition~\ref{eq: mlsi-conf}, we deduce that 
$$
\Ent_{\Psimple}(f)\leq  C_d' n\,  \Dir_{\Pconfig}(\tilde f, \log \tilde f),
$$
for any extension $\tilde f$. Here $C_d'>0$ is a constant depending only on $d$. 
We now choose a specific extension $\tilde f:\, \ConfBipGSet_n(d)\to \R_+$ for which $ \Dir_{\Pconfig}(\tilde f, \log \tilde f)$ can be compared to
$\Dir_{\Psimple}(f,\log f)$ and $\Ent_{\Psimple}\,f$. 
Recalling the notation $ \SNeigh(\cdot)$ from  Subsection~\ref{ajfnpakfjnwpfiqjwnfp},
we define 
$$
\tilde f(G'):=\begin{cases}
\frac{1}{\vert \SNeigh(G')\vert} \sum_{G\in \SNeigh(G')} f(G),&\mbox{if $G'\in \categ([0,\cconst])$}\\
\Exp_{\Psimple} f=1,&\mbox{otherwise}.\\
\end{cases}
$$
Using reversibility and the symmetry of $\Phi$, we can write 
\begin{align*}
\Dir_{\Pconfig}(\tilde f, \log \tilde f)
&=\frac12 \sum_{\underset{G_1'\sim G_2'}{G_1',G_2'\in \categ([0,\cconst])}} \Pconfig(G_1')Q_c(G_1',G_2') \Phi\big(\tilde f(G_1'), \tilde f(G_2')\big)\\
&+ \sum_{\underset{G_1'\sim G_2'}{G_1'\in \categ([0,\cconst]),\,G_2'\in
\categ([\cconst+1,\cconst+2])}} \Pconfig(G_1')Q_c(G_1',G_2') \Phi\big(\tilde f(G_1'), \tilde f(G_2')\big)\\
%&+\frac12 \sum_{\underset{G_1'\sim G_2'}{G_2'\in \categ([0,\cconst]),\,G_1'\in \categ([\cconst+1,\cconst+2])}} \Pconfig(G_1')Q_c(G_1',G_2') \Phi\big(\tilde f(G_1'), \tilde f(G_2')\big)\\
&+ \sum_{(G_1',G_2')\in W} \Pconfig(G_1')Q_c(G_1',G_2') \Phi\big(\tilde f(G_1'), \tilde f(G_2')\big)\\
%&+\frac12 \sum_{(G_2',G_1')\in W} \Pconfig(G_1')Q_c(G_1',G_2') \Phi\big(\tilde f(G_1'), \tilde f(G_2')\big),
\end{align*}
where by $W$ we denote the set
$$
W:=\big\{(G_1',G_2'):\;G_1'\sim G_2',\;G_1'\in \categ([0,\cconst]), \;G_2'\notin \categ([0,\infty))\big\}.
$$
In what follows, we estimate each of the terms above.

\medskip

Note that whenever $G_1',G_2'\in \categ([0,\cconst])$ and $G_1'\sim G_2'$, we have
$$
\tilde f(G_1')= \sum_{\underset{G_2\in \SNeigh(G_2')}{G_1\in \SNeigh(G_1')}} \beta_{G_1',G_2'}(G_1,G_2) f(G_1)\quad \text{ and }\quad \tilde f(G_2')= \sum_{\underset{G_2\in\SNeigh(G_2')}{G_1\in \SNeigh(G_1')}}  \beta_{G_1',G_2'}(G_1,G_2) f(G_2).
$$
Since $\Phi$ is convex as a function of two variables, we obtain that for any pair of such graphs $G_1',G_2'$,
$$
\Phi\big(\tilde f(G_1'), \tilde f(G_2')\big)\leq \sum_{\underset{G_2\in\SNeigh(G_2')}{G_1\in \SNeigh(G_1')}}\beta_{G_1',G_2'}(G_1,G_2) \Phi\big(f(G_1),f(G_2)\big).
%\;G_1',G_2'\in \categ([0,\cconst]),\;G_1'\sim G_2'.
$$
Using the last inequality and the definition of $\tilde Q_u$, we deduce
\begin{align*}
\frac12&
\sum_{\underset{G_1'\sim G_2'}{G_1',G_2'\in \categ([0,\cconst])}} \Pconfig(G_1')Q_c(G_1',G_2') \Phi\big(\tilde f(G_1'), \tilde f(G_2')\big)\\
&\hspace{1cm}\leq\frac12\sum_{
\substack{G_1',G_2'\in \categ([0,\cconst]):\\G_1'\sim G_2',\\G_1\in \SNeigh(G_1'),
G_2\in \SNeigh(G_2')}
} 
\Pconfig(G_1')Q_c(G_1',G_2')\beta_{G_1',G_2'}(G_1,G_2)\Phi\big(f(G_1),f(G_2)\big)\\
&\hspace{1cm}\leq 2 \sum_{G_1, G_2\in \BipGSet_n(d)} \Psimple(G_1) \tilde Q_u(G_1,G_2)  \Phi\big(f(G_1),f(G_2)\big)= 4\,\Dir_{\tilde Q_u,\Psimple}(f,\log f).
\end{align*}

\medskip

Further, consider the sum
$$
\sum_{\underset{G_1'\sim G_2'}{G_1'\in \categ([0,\cconst]),\,G_2'\in
\categ([\cconst+1,\cconst+2])}} \Pconfig(G_1')Q_c(G_1',G_2') \Phi\big(\tilde f(G_1'), \tilde f(G_2')\big).
$$
Note that any given $G_1'\in \categ([0,\cconst])$ has at most $nd(d-1)^2$ adjacent multigraphs $G_2'\in \categ([\cconst+1,\cconst+2])$. 
Using formulas \eqref{eq: s-neighborhood} and
\eqref{eq: number-neighborhood}, convexity of $\Phi$, and the definition of $Q_c$, we get
\begin{align*}
&\sum_{\substack{G_1'\sim G_2'\\G_1'\in \categ([0,\cconst])\\G_2'\in
\categ([\cconst+1,\cconst+2])}} 
\Pconfig(G_1')Q_c(G_1',G_2') \Phi\big(\tilde f(G_1'), \tilde f(G_2')\big)\\
&\hspace{2cm}\leq \frac{4d}{n}
\sum_{G_1'\in \categ([\cconst-1,\cconst])}
\Pconfig(G_1')\Phi\big(\tilde f(G_1'), 1\big)\\
&\hspace{2cm}\leq \frac{4d}{n}
\sum_{G_1'\in \categ(\cconst-1)}
\frac{\Pconfig(G_1')}{\vert \SNeigh(G_1')\vert} \sum_{G_1\in \SNeigh(G_1')}
\Phi\big(f(G_1), 1\big)\\
&\hspace{2.5cm}+\frac{4d}{n}
\sum_{G_1'\in \categ(\cconst)}
\frac{\Pconfig(G_1')}{\vert \SNeigh(G_1')\vert} \sum_{G_1\in \SNeigh(G_1')}
\Phi\big(f(G_1), 1\big)\\
&\hspace{2cm}\leq\frac{16d}{n}
\frac{(d-1)^{2\cconst}}{(\cconst-1)!}
\sum_{G_1\in \BipGSet_n(d)}
\Phi\big(f(G_1), 1\big)\Psimple(G_1)\\
&\hspace{2cm} \leq \frac{\tilde C_d}{n \log n}\Ent_{\Psimple}\,f,
\end{align*}
where in the last inequality we applied Lemma~\ref{nfapfienfpiwjfnpqifejnpqifjn}, our definition of $\cconst$ and that $n$ is large enough. Here also $\tilde C_d$ is a constant depending only on $d$. 
%By analogy, $$\sum_{\underset{G_1'\sim G_2'}{G_2'\in \categ([0,\cconst]),\,G_1'\in \categ([\cconst+1,\cconst+2])}}  \Pconfig(G_1')Q_c(G_1',G_2') \Phi\big(\tilde f(G_1'), \tilde f(G_2')\big) = o\bigg(\frac{C_d}{n}\bigg)\Ent_{\Psimple}\,f.$$

\medskip

Consider now the sum
$$
\sum_{(G_1',G_2')\in W} \Pconfig(G_1')Q_c(G_1',G_2') \Phi\big(\tilde f(G_1'), \tilde f(G_2')\big).
$$
Note that we always have $Q_c(G_1',G_2') \leq 4n^{-2}$ provided $n$ is large enough. 
%whenever $(G_1',G_2')\in W$.
Further, $(G_1',G_2')\in W$ only if the graph $G_2'$ is obtained from $G_1'$
by either adding an edge of multiplicity three or introducing a multiedge incident to one
of the existing multiedges in the graph. This implies that for every $G_1'\in \categ([0,\cconst])$,
the number of graphs $G_2'$ such that $(G_1',G_2')\in W$, is at most $c_d'\cconst$ for some constant $c_d'$ depending only on $d$.
Thus, using formulas \eqref{eq: s-neighborhood} and
\eqref{eq: number-neighborhood}, convexity of $\Phi$, and the definition of $\Pconfig$, we can write
\begin{align*}
\sum_{(G_1',G_2')\in W} \Pconfig(G_1')Q_c(G_1',G_2') \Phi\big(\tilde f(G_1'), \tilde f(G_2')\big)
&\leq \frac{4c_d'\cconst}{n^2}\sum_{G_1'\in \categ([0,\cconst])}\Pconfig(G_1')
\Phi\big(\tilde f(G_1'), 1\big)\\
&\leq \frac{8c_d'\cconst}{n^2}\sum_{k=0}^{\cconst}\frac{(d-1)^{2k}}{2^k k!}
\sum_{G_1\in \BipGSet_n(d)}
\Pconfig(G_1)\Phi\big(f(G_1), 1\big)\\
&\leq \frac{16 c_d'\cconst}{n^2}
\sum_{G_1\in \BipGSet_n(d)}
\Phi\big(f(G_1), 1\big)\Psimple(G_1)\\
&\leq \frac{C_d''\cconst}{n^2}\Ent_{\Psimple}\,f,
\end{align*}
where the last inequality follows from Lemma~\ref{nfapfienfpiwjfnpqifejnpqifjn} with some constant $C_d''$ depending only on $d$.
%We estimate the sum $$\sum_{(G_2',G_1')\in W} \Pconfig(G_1')Q_c(G_1',G_2') \Phi\big(\tilde f(G_1'), \tilde f(G_2')\big)$$ by analogy.

\bigskip

Combining the above estimates, we obtain
$$
\Dir_{\Pconfig}(\tilde f, \log \tilde f)
\leq 4\,\Dir_{\tilde Q,\Psimple}(f,\log f)
+\frac{\tilde c_d}{n\log n}\Ent_{\Psimple}\,f,
$$
whence
$$
\Ent_{\Psimple}(f)\leq C_d'''n \Bigg(\Dir_{\tilde Q,\Psimple}(f,\log f)
+\frac{\tilde c_d}{n\log n}\Ent_{\Psimple}\,f\Bigg),
$$
for some constants $\tilde c_d$ and $C_d'''$ depending only on $d$. 
The result follows.

%$$
%\Dir_{\Pconfig}(\tilde f, \log \tilde f)
%\leq \sum_{G_1, G_2\in \BipGSet_n(d)} \Psimple(G_1) \tilde Q(G_1,G_2)  \Phi\big(f(G_1),f(G_2)\big) +%\textcolor{red}{another term}.
%$$
%\textcolor{red}{While the first term is okay, it is not clear what should be the extension on the other graphs. If we define it to be $\Exp_{\Psimple} f$ say. Then 
%We would get a term of the form 
%$$
%\sum_{\underset{G_1'\sim G_2'}{G_1'\in \categ([0,\cconst]), G_2'\not\in \categ([0,\cconst])}} %\Pconfig(G_1')Q_c(G_1',G_2') \Phi\big(\tilde f(G_1'), \Exp_{\Psimple} f\big).
%$$
%This can be much larger than entropy and at the same time it is not clear how to compare it to the Dirichlet form. 
%}
\end{proof}

\subsection{Proof of Theorem~\ref{th: mlsi-switch}}

The strategy of the proof is to apply the comparison
Theorem~\ref{th: comparison-mlsi} with %the switch chain on $\BipGSet_n(d)$ to 
the auxiliary chain $(\Psimple,\tilde Q_u)$ defined in \eqref{eq: def-aux-chain}. 
To this aim, we will use the family of paths introduced in Subsection~\ref{sec: paths} in order to define a $(Q_u, \tilde Q_u)$-flow. In what follows, for every $4$--tuple $\tuple=(G_1,G_1',G_2,G_2')$
such that $G_1',G_2'\in \categ([0,\cconst])$, $G_1'\sim G_2'$, and $G_1\in \SNeigh(G_1')$,
$G_2\in \SNeigh(G_2')$, $G_1\neq G_2$, we write $P_{\tuple}$ for
\begin{itemize}
    \item the [trivial] path of length one from $G_1$ to $G_2$ in the case when $(G_1',G_2')\in \Perfmatch$ and $G_2=\psi_{G_1',G_2'}(G_1)$;
    \item the empty path, when $(G_1',G_2')\in \Perfmatch$ and $G_2\neq\psi_{G_1',G_2'}(G_1)$;
    \item the connection $P_{\tuple}$ from the statement of Proposition~\ref{prop: connections} when $\tuple \in \tupleset$.
\end{itemize}

\begin{defi}
Consider the two Markov generators $Q_u$ and $\tilde Q_u$ on $(\BipGSet_n(d), \Psimple)$ and define $\weight:\, \Gamma(Q_u,\tilde Q_u)\to [0,1]$ as follows. 
Given $G_1, G_2\in \BipGSet_n(d)$ with $\tilde Q_u(G_1,G_2)>0$ and a valid $Q_u$-path $P$ between $G_1$ and $G_2$, we set 
$$\weight(P):=\frac14\sum_{\tuple=(G_1,G_1',G_2,G_2'):\,P=P_{\tuple}} \Pconfig(G_1')Q_c(G_1',G_2')\beta_{G_1',G_2'}(G_1,G_2).$$ 
%if there exists $\tuple=(G_1,G_1',G_2,G_2')$ such that $P=P_{\tuple}$; and $\weight(P)=0$ otherwise. 
\end{defi}
In view of the definition of $\tilde Q_u$ in \eqref{eq: def-aux-chain}, for every $G_1, G_2\in \BipGSet_n(d)$ with $\tilde Q_u(G_1,G_2)>0$,
\begin{align*}
\sum_{\substack{P\mbox{\tiny{ valid $Q_u$-path}}\\\mbox{\tiny{between $G_1$ and $G_2$}}}} \weight(P)
&=
\frac14\sum_{\tuple=(G_1,G_1',G_2,G_2')}\Pconfig(G_1')Q_c(G_1',G_2')\beta_{G_1',G_2'}(G_1,G_2)\\
&= \Psimple(G_1) \tilde Q_u(G_1,G_2),
\end{align*}
so that the weight function is indeed a $(Q_u, \tilde Q_u)$-flow. 
In order to make use of Theorem~\ref{th: comparison-mlsi}, we need to calculate a version of the flow congestion.
The following lemma helps in this respect.

\begin{lemma}\label{lem: flow-congestion}
Consider the two Markov generators $Q_u$ and $\tilde Q_u$ on $(\BipGSet_n(d), \Psimple)$, and let 
$\weight$ be the $(Q_u, \tilde Q_u)$-flow defined above. 
Let $H, H'\in \BipGSet_n(d)$ be such that $Q_u(H,H')>0$. Then for any $t\geq 1$, 
$$
\sum_{\underset{(H,H')\in \path}{\path\in \Gamma(Q_u,\tilde Q_u):}} \weight(\path)\big(1+ (\length{\path}-1)^2t\big)
\leq C\big(1+\frac{t}{\sqrt{n}}\big) \Psimple(H) Q_u(H,H'),
$$
for some universal constant $C>0$. 
\end{lemma}
\begin{proof}
Denote 
$$
\delta:=\sum_{\underset{(H,H')\in \path}{\path\in \Gamma(Q_u,\tilde Q_u):}} \weight(\path)\big(1+ (\length{\path}-1)^2t\big).
$$
In view of the definition of $\weight$, we have 
$
\delta= \frac12\sum_{\tuple:\, (H,H')\in \path_{\tuple}}\delta_{\tuple}$, where for $\tuple=(G_1,G_1',G_2,G_2')$ we defined 
$$
\delta_\tuple:= \Pconfig(G_1')Q_c(G_1',G_2')\beta_{G_1',G_2'}(G_1,G_2)\big(1+ (\length{\path_\tuple}-1)^2t\big). 
$$
Now note that whenever $(G_1',G_2')\in \Perfmatch$ and $(H,H')\in \path_{\tuple}$, necessarily $G_1=H$, $G_2=H'$, and $\length{\path_\tuple}=1$. 
Therefore, we have 
\begin{align*}
\sum_{\underset{\underset{H'=\psi_{G_1',G_2'}(H)}{(G_1',G_2')\in \Perfmatch}}{\tuple=(H,G_1',H',G_2')}}\delta_\tuple &=\sum_{\substack{(G_1',G_2')\in \Perfmatch\\H\in \SNeigh(G_1'),H'\in \SNeigh(G_2')\\H'=\psi_{G_1',G_2'}(H)}}
\frac{\Pconfig(G_1')Q_c(G_1',G_2')}{\vert \SNeigh(G_1')\vert}\\
& \leq 2e^{-\frac{(d-1)^2}{2}} \Psimple(H) Q_u(H,H')\sum_{%\underset{H\in \SNeigh(G_1'), %H'\in \SNeigh(G_2')
%\co{H'=\psi_{G_1',G_2'}(H)}}{(G_1',G_2')\in \Perfmatch}
\substack{(G_1',G_2')\in \Perfmatch\\H\in \SNeigh(G_1'),H'\in \SNeigh(G_2')\\H'=\psi_{G_1',G_2'}(H)}}
\frac{1}{\vert \SNeigh(G_1')\vert},
\end{align*}
where we made use of \eqref{eq: size-multigraph}. Now applying the second point of Proposition~\ref{prop: matching}, we deduce that 
$$
\sum_{\underset{\underset{H'=\psi_{G_1',G_2'}(H)}{(G_1',G_2')\in \Perfmatch}}{\tuple=(H,G_1',H',G_2')}}\delta_\tuple\leq 2e^{-\frac{(d-1)^2}{2}} \Psimple(H) Q_u(H,H') \sum_{G_1'\in \categ([1,\cconst]):\, H\in \SNeigh(G_1')}
\frac{1}{\vert \SNeigh(G_1')\vert}.
$$
Making use of \eqref{eq: s-neighborhood} and \eqref{eq: number-neighborhood}, we get that 
$$\sum_{\underset{\underset{H'=\psi_{G_1',G_2'}(H)}{(G_1',G_2')\in \Perfmatch}}{\tuple=(H,G_1',H',G_2')}}\delta_\tuple\leq 4 \Psimple(H) Q_u(H,H').$$ 

On the other hand, using Proposition~\ref{prop: connections}, we have 
$$
\sum_{\tuple\in \tupleset:\, (H,H')\in \path_{\tuple}}\delta_{\tuple}\leq t\sum_{\underset{(H,H')\in \path_\tuple}{\tuple =(G_1,G_1',G_2,G_2') \in \tupleset}} \frac{\Pconfig(G_1') Q_c(G_1',G_2')}{\vert \SNeigh(G_1')\vert\, \vert \SNeigh(G_2')\vert} \length{\path_\tuple}^2\leq C\frac{t\Psimple(H) Q_u(H,H')}{\sqrt{n}},
$$
for some universal constant $C$. Combining the above estimates, we finish the proof. 
\end{proof}

\begin{proof}[Proof of Theorem~\ref{th: mlsi-switch}]
Without loss of generality, $n\geq n_d$ where $n_d$ is taken from Proposition~\ref{prop: auxiliary chain}.
Combining Theorem~\ref{th: comparison-mlsi}, Proposition~\ref{prop: auxiliary chain} and  Lemma~\ref{lem: flow-congestion}, we deduce that 
for any $r\geq e$,  $(\BipGSet_n(d), \Psimple, Q_u)$ satisfies an $r$-regularized Modified log-Sobolev Inequality with a constant $C_d\big(1+\frac{\log r}{\sqrt{n}}\big)n$. Here $C_d>0$ depends only on $d$. 
It remains to apply Theorem~\ref{th: MLSI=reg-MLSI} to finish the proof. 
\end{proof}

\end{document}